\documentclass{amsart}
\usepackage{amssymb, pb-diagram}
\usepackage{graphicx}

\usepackage{bbold}

\usepackage{appendix}
\usepackage{amsthm}

\usepackage{appendix}
\usepackage{mathtools}

\def\mtline#1{\hbox to#1{\hrulefill}}

\def\what{\widehat}

\def\maB{{\mathcal B}}

\def\maE{{\mathcal E}}

\def\maG{{\mathcal G}}

\def\maK{{\mathcal K}}
\def\maL{{\mathcal L}}
\def\maM{{\mathcal M}}

\def\maP{{\mathcal P}}
\def\maR{{\mathcal R}}
\def\maS{{\mathcal S}}

\def\maU{{\mathcal U}}

\def\C{{\mathbb C}}

\def\N{{\mathbb N}}

\def\R{{\mathbb R}}

\def\Z{{\mathbb Z}}

\let\epr=\endproclaim

\usepackage{secteqn}
\usepackage{mathabx}

\usepackage{color}
\definecolor{purple}{cmyk}{.33,1,0,.4}
\definecolor{m}{rgb}{1,0.1,1}
\definecolor{violet}{cmyk}{1,0,1,0}
\definecolor{test}{rgb}{1,1,1}
\definecolor{cmyk}{cmyk}{0,1,1,0}

\newtheorem{Equation}{}[section]

\newtheorem{conjecture}[Equation]{Conjecture}
\newtheorem{corollary}[Equation]{Corollary}
\newtheorem{definition}[Equation]{Definition}
\newtheorem{example}[Equation]{Example}

\newtheorem{lemma}[Equation]{Lemma}
\newtheorem{proposition}[Equation]{Proposition}

\newtheorem{remark}[Equation]{Remark}

\newtheorem{theorem}[Equation]{Theorem}

\newcommand\dirac{\partialslash}

\def\dim{\operatorname{dim}}

\def\Id{\operatorname{I}}
\def\id{\operatorname{id}}

\def\Ker{\operatorname{Ker}}

\def\Supp{\operatorname{Supp}}

\def\ep{\epsilon}

\def\bpr{\begin{proof}}
\def\epr{\end{proof}}

\def\pa{\partial}

\def\ms{\medskip}

\def\what{\widehat}

\usepackage{hyperref}

\usepackage[colorinlistoftodos]{todonotes}
\usepackage[english]{babel}
\usepackage[utf8x]{inputenc}
\usepackage[all]{xy}
\usepackage{amscd}

\def\Ad{\operatorname{Ad}}

\def\Id{\operatorname{Id}}
\def\id{\operatorname{id}}

\def\Ker{\operatorname{Ker}}

\def\Supp{\operatorname{Supp}}

\def\Sp{\operatorname{Sp}}

\def\1{{\mathbb 1}}
\def\C{\mathbb C}

\def\N{\mathbb N}
\def\R{\mathbb R}

\def\Z{\mathbb Z}

\def\maL{\mathbb L}

\def\maB{{\mathcal B}}

\def\maE{{\mathcal E}}

\def\maL{{\mathcal L}}
\def\maM{{\mathcal M}}

\def\maG{{\mathcal G}}

\def\maK{{\mathcal K}}

\def\maP{{\mathcal P}}

\def\maR{{\mathcal R}}

\def\maS{{\mathcal S}}

\def\maU{{\mathcal U}}

\def\what{\widehat}

\def\tE{\widetilde{E}}

\def\tM{\widetilde{M}}

\def\pa{\partial}

\def\pphi{\varphi}
\begin{document}



\title
[Admissible HR sequence for groupoids]{Admissible Higson-Roe sequences\\ for transformation groupoids\\ \today}


\author[M.-T. Benameur]{Moulay Tahar Benameur}
\address{Institut Montpellierain Alexander Grothendieck, UMR 5149 du CNRS, Universit\'e de Montpellier}
\email{moulay.benameur@umontpellier.fr}

\author[V. Moulard]{Victor Moulard}
\address{Institut Montpellierain Alexander Grothendieck, UMR 5149 du CNRS, Universit\'e de Montpellier} 
\email{victor.moulard@umontpellier.fr}

\thanks{MSC (2010) 53C12, 57R30, 53C27, 32Q10, 19K33, 19K35, 19K56, 58B34. \\
Key words: $K$-theory of $C^*$-algebras, Actions of discrete groups, foliation, positive scalar curvature, Baum-Connes conjecture, Higson-Roe  sequence.}

%
\maketitle

%
%
%


%
%

%
\begin{abstract} Given a finitely generated discrete group $\Gamma$, we construct for any admissible crossed product completion and for any metrizable finite dimensional compact $\Gamma$-space $X$, a universal Higson-Roe six-term exact sequence for the transformation groupoid $X\rtimes \Gamma$. In particular, we generalize the maximal Higson-Roe sequence of \cite{HigsonRoe2008} to such groupoids. In the case where the groupoid $X\rtimes \Gamma$  satisfies the rectified Baum-Connes conjecture of \cite{BaumGuentnerWillett}, this yields some rigidity consequences. 
 \end{abstract}

\medskip

%



\section*{Introduction}

Given a smooth closed odd dimensional manifold $M$, Atiyah, Patodi and Singer  \cite{APS3} associated with any generalized Dirac operator $D$ on $M$ and any  finite dimensional unitary representation $\sigma$ of  the fundamental group $\Gamma=\pi_1(M)$, a real  invariant $\rho (D, \sigma)$  called since then the {\em APS rho-invariant}. This  is essentially the difference between the eta invariant associated with the twisted operator $D\otimes \sigma$ and $\dim (\sigma)$ times the eta invariant of $D$. Given  more generally any Galois covering over $M$ with group $\Gamma$, Cheeger and Gromov  \cite{CheegerGromov} introduced an $\ell^2$ version of the eta invariant, as well as an $\ell^2$  rho-invariant. They replaced the finite dimensional representation $\sigma$ by the  regular representation in the Hilbert space $\ell^2(\Gamma)$ so that  whenever the covering happens to be profinite, their $\ell^2$ eta invariant  is the expected  limit of APS invariants. In general, the Cheeger-Gromov  $\ell^2$ eta invariant is defined   using the Atiyah trace on the group von Neumann algebra \cite{AtiyahCovering} and the resulting $\ell^2$ rho invariant, also called the Cheeger-Gromov (or CG) rho invariant and denoted $\rho_{(2)} (D)$, is then the difference between the $\ell^2$ eta invariant and the APS eta invariant of $D$.

The first important observation is that when  $D$ is the signature operator associated with a riemannian metric  on $M$, the APS rho invariant and the CG rho invariant do not depend on such metric  \cite{APS3, CheegerGromov} and are  differential invariants.  In \cite{Neumann} Neumann   proved  that the APS rho invariant is even an oriented homotopy invariant when  the representation $\sigma$ factors through some free abelian group. Some years later, Weinberger  \cite{Weinberger88} proved the same homotopy invariance for a large class of torsion-free groups and conjectured that the APS rho invariant for the signature operator should be an oriented  homotopy invariant for all torsion-free finitely generated countable discrete groups. Mathai \cite{Mathai} proved similarly  the oriented homotopy invariance of the $\ell^2$ rho invariant   for torsion-free Bieberbach groups, and computed it for some locally symmetric spaces. He also  conjectured  that the Cheeger-Gromov rho invariant for the signature operator should also be  an oriented homotopy invariant for any torsion-free discrete countable group. So, in this torsion-free case the rho-invariants seem to behave like signatures. On the other hand, when the group has torsion then  the APS invariants as well as the CG invariant have  totally different behaviour. They were for instance used to distinguish homotopy invariant manifolds. The CG invariant was  used to prove the existence, for some manifolds $M$ whose fundamental group has torsion,  of infinitely many differential manifolds which are oriented homotopy equivalent  to $M$ \cite{ChangWeinberger}. 

Other important properties of the APS and CG rho-invariants are used in the study of metrics of Positive Scalar Curvature (PSC).  Indeed, on a given spin manifold which admits a metric $g$ of PSC, the rho invariants for the spin-Dirac operator $D=\dirac_g$ only depend on the PSC-concordance class of $g$, and hence provide invariants of the space of path-connected components of the space of metrics of PSC on $M$, or rather of its quotient under diffeomorphisms \cite{Gilkey, PiazzaSchick}. Here again, there is a complete dichotomy between the torsion-free case and the non torsion-free case. So, the rho invariants seem to behave like the index of the corresponding Dirac operator as far as the group is torsion-free, i.e.  to vanish when the metric has PSC \cite{Keswani1, Keswani2, PiazzaSchick}. But when the group has torsion, they provide in contrast   interesting non-trivial invariants \cite{Gilkey}. More recent important contributions in this direction can be found in the interesting paper \cite{WeinbergerYu}. Other results were obtained for topological manifolds  in   \cite{ZenobiTop}, see also \cite{WeinbergerXiYu}.

In the late 90's, the rigidity of the rho invariants for torsion-free groups  was approached  using $K$-theory of $C^*$-algebras in the phD thesis of  Keswani  \cite{Keswani1, Keswani2}. More precisely, he related these questions to the maximal Baum-Connes assembly map and he proved the corresponding rigidity results under the further condition that the maximal Baum-Connes assembly map for $\Gamma$ is an isomorphism. 
%
%
In \cite{HigsonRoe2008}, the Keswani theorems were clarified by Higson and Roe by  using  their universal analytic structure group  $\maS_1(\Gamma)$ from their seminal three papers on ``mapping surgery to analysis'' \cite{HigsonRoeAnalysis1, HigsonRoeAnalysis2, HigsonRoeAnalysis3}. More precisely,  Higson and Roe extended their universal analytic exact sequence to include maximal completions and  gave a  conceptual explanation of the rigidity results of the APS rho invariants in Keswani's theorems \cite{HigsonRoe2008}. Recall that the maximal Higson-Roe  universal sequence can be written, when  the group $\Gamma$ is for instance torsion-free, as follows
 $$\xymatrix{
K_{0} (B\Gamma) \ar@{->}[r]^{\mu_{0, \Gamma}\hspace{0,3cm}}\ar@{<-}[d] &
K_0(C^*(\Gamma))\ar@{->}[r]&
\maS_1(\Gamma)\ar@{->}[d]\\
\maS_0(\Gamma)\ar@{<-}[r]&
K_1(C^*(\Gamma))\ar@{<-}[r]^{\hspace{0,3cm}\mu_{1, \Gamma}}&
K_{1} (B\Gamma) 
}$$
where $\maS_1(\Gamma)$ is an appropriately modified universal structure group associated with $\Gamma$ which is  constructed  using maximal completions. Here $K_{*} (B\Gamma)$ is as usual the compactly supported  $K$-homology group of the classifying space $B\Gamma$ and  the map  $\mu_{i,\Gamma}:  K_{i} (B\Gamma)\stackrel{\mu_\Gamma}{\rightarrow} K_i(C^*(\Gamma))$ is the the maximal Baum-Connes assembly map, where $K_*(C^*(\Gamma))$ is the $K$-theory groups of the full $C^*$-algebra of $\Gamma$. Once this maximal sequence is carried out, they  achieved their proof of the Keswani rigidity theorems by constructing  a group morphism $\maS_1(\Gamma)\to \R$ associated with any finite dimensional unitary representation $\sigma$ of $\Gamma$,  which allowed them to recover  the main  APS rho invariants associated with $\sigma$ as elements of the range of  this morphism. Keswani's theorem hence became a ``Lapalissade'': whenever $\maS_1(\Gamma)$ is trivial,  its image under the group morphism  $\maS_1(\Gamma)\to \R$ is  trivial.  
In \cite{BenameurRoyJFA}, the Cheeger-Gromov rho invariant was also encompassed by a similar approach,  now constructing a  morphism $\maS_1(\Gamma)\to \R$ associated with the regular representation, and using von Neumann algebras.  Again, the main $\ell^2$ rho invariants were proved to belong to the range of this group morphism, and hence allowed again to deduce the  rigidity results for the CG invariants, under the same hypothesis on the maximal Baum-Connes assembly map, see again \cite{BenameurRoyJFA}. \\

The hypothesis on the maximal Baum-Connes assembly map  excludes  many interesting classes of groups having Kazhdan's property (T), and for which the (reduced) Baum-Connes conjecture is know to be true  \cite{Lafforgue}.   Taking for instance $\Gamma$ to be the fundamental group of a closed hyperbolic $3$-manifold or any lattice in $\Sp(n,1)$,  the above $K$-theory approach doesn't apply, see \cite{Lafforgue} or \cite{Julg}. So, the rigidity conjectures  for such groups (Weinberger and Mathai) look unreachable by  this method. 
These observations motivated the intensive study of transformation groupoids in  \cite{BenameurRoyI, BenameurRoyII, BenameurRoyPPV}. More precisely,  given that the group $\Gamma$ usually admits interesting actions  on compact spaces, which do satisfy the maximal Baum-Connes conjecture even if $\Gamma$ doesn't, it was necessary to extend the construction of the Higson-Roe sequence and its consequences, so as to encompass such transformation groupoids.  Zimmer-amenable actions on finite dimensional metrizable compact spaces give new examples, e.g. amenable actions on smooth compact manifolds, in this case the maximal Baum-Connes map is an isomorphism by the Higson-Kasparov theorem \cite{HigsonKasparov}, and the results obtained in \cite{BenameurRoyII}   can already be applied to yield new proofs of the rigidity results obtained in \cite{BenameurPiazza}. On the other hand, stricking examples have been discovered in the last two decades where such   $\Gamma$-actions on  compact spaces $X$ do  satisfy the maximal Baum-Connes conjecture without satisfying the  reduced Baum-Connes conjecture, see for instance \cite{HigsonLafforgueSkandalis}[Remark 12] and  also \cite{Osajda}. These examples  motivated the modification of the RHS of the Baum-Connes assembly map so as to impose exactness of the chosen crossed product completion and to state a rectified Baum-Connes conjecture in \cite{BaumGuentnerWillett}.\\

The goal of the present paper is to provide  the  universal analytic six-term  sequence for  transformation groupoids like $X\rtimes \Gamma$  which is valid for  any admissible crossed-product completion, with the constraint of involving  the corresponding  Baum-Connes assembly map. In particular, our results apply to the  minimal exact and Morita compatible completion considered  in \cite{BaumGuentnerWillett} and hence involving the rectified Baum-Connes assembly map defined there, but we also encompass    the maximal crossed-product completion which thus involves the maximal assembly map and extends the results of \cite{HigsonRoe2008}. 

Let us now describe  in more details  our results.  
We first need  to define dual algebras corresponding to admissible completions,  keeping in mind that the resulting universal sequence should  involve the corresponding Baum-Connes assembly map. An admissible completion will be for us any crossed product functor $\rtimes$ in the sense of \cite{BaumGuentnerWillett}[Definition 2.1] which is Morita compatible. More precisely, for any separable $\Gamma$-algebra $A$, $A\rtimes\Gamma$ contains $A\rtimes_{alg}\Gamma$ as a dense subalgebra, and  there are natural transformations 
$$
A\rtimes_{\max}\Gamma \to A\rtimes \Gamma \to A\rtimes_{r}\Gamma.
$$
which restrict to the identity on $A\rtimes_{alg}\Gamma$. Moreover, we assume for simplicity the Morita compatibility condition in the whole paper, although many results are obviously true without such assumption. However, an admissible crossed product need not to be exact a priori, note though that the exactness property is necessary for the Baum-Connes assembly map to be an isomorphism \cite{HigsonLafforgueSkandalis}. Since all the Baum-Connes assembly maps share the same LHS, namely  $\maR K^0(\maG)$ (only the RHS uses different completions),  our definition of the dual Roe algebras and their ideals needs to keep the same $K$-groups for the quotient dual algebra, exactly as in \cite{HigsonRoe2008} for the maximal completion and for groups. Following \cite{HigsonRoe2008}, we use the concept of lifts of operators, and our dual algebras will be composed of lifts of operators from the above reduced algebras. This is achieved using a generalized Connes-Skandalis Hilbert module which is obtained as the range of the so-called Mishchenko projection. More precisely, given a locally compact metric $\Gamma$-space $Z$ which is proper and cocompact, and a faithful non-degenerate $\Gamma$-equivariant $C(X)$-representation $\pi$ of $C_0(Z\times X)$ in $C(X)\otimes H$ amplified using $\what{H}=H\otimes \ell^2(\Gamma)^\infty$, and working with a fixed  admissible crossed product functor $\rtimes$ in the sense of \cite{BaumGuentnerWillett}, we consider the Hilbert $C(X)\rtimes\Gamma$-module  $[C(X)\otimes\what{H}]\rtimes\Gamma$, and hence its Hilbert submodule $\maE_Z^{\what{H}}$ obtained as the range of a Mishchenko projection associated with some compactly supported continuous cut-off function on $Z\times X$.  This latter module is called the generalized Connes-Skandalis Hilbert module,  it coincides in the simplest situations with the Hilbert module defined for foliations in \cite{ConnesSkandalis}, see also \cite{BenameurPiazza}.

The generalized Connes-Skandalis Hilbert module has another description as a composition Hilbert module, therefore we  define the dual Roe algebra as some $C^*$-subalgebra $D^*_\Gamma (\maE_Z^{\what{H}})$ of the $C^*$-algebra 
of adjointable operators in the generalized Connes-Skandalis Hilbert module $\maE_Z^{\what{H}}$. The Roe ideal $C^*_\Gamma (\maE_Z^{\what{H}})$ is also  defined similarly, and we end up  with a short exact sequence of Roe algebras:
$$
0\to C^*_\Gamma (\maE_Z^{\what{H}}) \hookrightarrow D^*_\Gamma (\maE_Z^{\what{H}}) \rightarrow Q^*_\Gamma (\maE_Z^{\what{H}})\to 0.
$$
The first observation is  that the ideal $C^*_\Gamma (\maE_Z^{\what{H}})$ is isomorphic to  the  ideal of all compact operators on $\maE_Z^{\what{H}}$. Therefore, $C^*_\Gamma (\maE_Z^{\what{H}})$ is Morita equivalent to the $C^*$-algebra $C^*(\maG)=C(X)\rtimes\Gamma$  corresponding to the fixed admissible crossed product. Next we prove  that there is a well defined $C^*$-algebra morphism 
$$
Q^*_\Gamma (X; Z, \what{H})\longrightarrow Q^*_\Gamma (\maE_Z^{\what{H}}),
$$
which turns out to be an isomorphism. Here $Q^*_\Gamma (X; Z, \what{H})$ is the Calkin $L^2$ Roe algebra corresponding to the reduced completion, and whose $K$-theory was computed in \cite{BenameurRoyII}. Therefore, the isomorphism class of the $C^*$-algebra $Q^*_\Gamma (\maE_Z^{\what{H}})$ does not depend on the completion and gives back the LHS of the Baum-Connes map  as allowed. Finally, we use the equivariant version of the Pimsner-Popa-Voiculescu theorem \cite{PPV1, PPV2} as stated in \cite{BenameurRoyPPV} to prove  that the $K$-theory groups of all these dual $C^*$-algebras do not depend on the choice of the above non-degenerate $\Gamma$-equivariant representation $\pi$ of $C_0(Z\times X)$. Our first theorem can then be stated as follows:
\begin{theorem}
For any metric proper cocompact  $\Gamma$-space $Z$, there exists a  periodic  six-term exact sequence:
$$\xymatrix{
KK^0_{\Gamma} (Z, X) \ar@{->}[r]^{\mu^\maG_{0, Z}\hspace{0,3cm}}\ar@{<-}[d] &
K_0(C(X)\rtimes\Gamma)\ar@{->}[r]&
K_0(D^*_\Gamma (\maE_Z^{\what{H}}))\ar@{->}[d]\\
K_1(D^*_\Gamma (\maE_Z^{\what{H}}))\ar@{<-}[r]&
K_1(C(X)\rtimes\Gamma)\ar@{<-}[r]^{\hspace{0,3cm}\mu^\maG_{1, Z}}&
KK^1_{\Gamma} (Z, X) 
}$$
where $\mu^\maG_{i, Z}$ is the Baum-Connes index map associated with the proper cocompact $\Gamma$-space $Z$ with coefficients in the $\Gamma$-algebra $C(X)$, defined using the chosen admissible completion. 
\end{theorem}

When we apply this theorem to the minimal admissible completion, say the reduced crossed product, we get the exact sequence obtained in \cite{BenameurRoyII}. On the other hand, using the maximal admissible completion, say the full crossed product, we get the allowed maximal Higson-Roe sequence studied  in  \cite{VictorThesis}, which in turn gives back  the maximal Higson-Roe sequence obtained in \cite{HigsonRoe2008} when  $X=\{\bullet\}$. If we use the minimal exact and Morita compatible completion of \cite{BaumGuentnerWillett} then we obtain a six-term exact sequence involving the rectified Baum-Connes index map associated with $Z$, with coefficients in $C(X)$. 

Applying  the main theorem of \cite{BenameurRoyPPV}, we  prove that the above six-term sequence  is functorial in $Z$. This part is technically more involved than in the reduced case, but all the statements coincide with the expected ones. More precisely, we show that if 
$\iota:Z\hookrightarrow Z'$ be a $\Gamma$-inclusion  of the proper cocompact  metric $\Gamma$-spaces ($Z$ is closed in $Z'$), and  $\pi:C_0(Z)\to \maL_{C(X)}(C(X)\otimes H)$ and $\pi':C_0(Z')\to \maL_{C(X)}(C(X)\otimes H')$ are two non-degenerate faithful representations, then there is  a well defined $\Z_2$-graded functoriality morphism
$$
\iota_{\pi, \pi'}^{Z, Z'} : K_* (D^*_\Gamma (\maE_Z^{\what H})) \longrightarrow K_* (D^*_\Gamma (\maE_{Z'}^{{\what H}'})),
$$
and an induced morphism $K_* (Q^*_\Gamma (\maE_Z^{\what H})) \to K_* (Q^*_\Gamma (\maE_{Z'}^{{\what H}'}))$  on the quotient algebras which is compatible with the functoriality morphism $K_*(Q^*_\Gamma (X; Z, \what{H}))\to K_*(Q^*_\Gamma (X; Z', \what{H}'))$  defined in \cite{BenameurRoyII}. Moreover, these morphisms only depend, up to  conjugation by isomorphisms, on the $\Gamma$-pair $Z\hookrightarrow Z'$. Finally,  if $ Z'\hookrightarrow Z''$ is another $\Gamma$-inclusion, then $
{\iota}_{\pi', \pi''}^{Z', Z''} \circ \iota_{\pi, \pi'}^{Z, Z'} = {\iota}_{\pi, \pi''}^{Z, Z''}$.
A corollary of all these functoriality results is that we can define  universal analytic structure groups $\maS_*(\Gamma, X)$ for any admissible crossed product completion, and that these  fit in a universal periodic six-term exact sequence corresponding to this fixed admissible completion. More precisely, 

\begin{theorem}
Given a finitely generated countable discrete  group $\Gamma$ and a compact metrizable finite dimensional $\Gamma$-space $X$, there exists for any admissible crossed product completion  a universal analytic six-term exact sequence
$$\xymatrix{
RK_{0, \Gamma} ({\underline{E}}\Gamma, X) \ar@{->}[r]^{\mu_{0, \maG}\hspace{0,3cm}}\ar@{<-}[d] &
K_0(C(X)\rtimes\Gamma)\ar@{->}[r]&
\maS_1(\Gamma, X)\ar@{->}[d]\\
\maS_0(\Gamma, X)\ar@{<-}[r]&
K_1(C(X)\rtimes\Gamma)\ar@{<-}[r]^{\hspace{0,3cm}\mu_{1, \maG}}&
RK_{1, \Gamma} ({\underline{E}}\Gamma, X) 
}$$
\end{theorem}

\medskip

When the isotropy groups of the $\Gamma$-action are torsion-free, we can replace the universal group $RK_{0, \Gamma} ({\underline{E}}\Gamma, X)$ by the compactly supported $K$-homology of the classifying CW-complex $B\maG$ of the transformation groupoid $\maG=X\rtimes\Gamma$. An obvious consequence of our exact sequence applied to the Baum-Guentner-Willett crossed product completion,  is that the rectified Baum-Connes conjecture for $\Gamma$ with coefficients in $C(X)$ \cite{BaumGuentnerWillett}  is satisfied if and only if the corresponding analytic structure groups  vanish, hence in this case all invariants that would be extracted from these structure groups will vanish as well. Notice also that the maximal  universal periodic sequence  is functorial in $X$, so given any countable discrete finitely generated group $\Gamma$, there exists for any  finite dimensional compact $\Gamma$-space $X$, a morphism of exact sequences from the maximal Higson-Roe sequence for the group $\Gamma$ \cite{HigsonRoe2008} to that for $X\rtimes \Gamma$, which is the identity morphism when $X$ is reduced to $\{\bullet\}$. In \cite{VictorThesis}, a morphism sending the maximal sequence obtained above to a short exact sequence involving $\ell^2$ dual algebras and von Neumann algebras, with traces associated with an invariant measure (when such data does exist), is constructed and yields to connections with rigidity of measured rho invariants as studied in \cite{BenameurPiazza}, see also \cite{BenameurPreprint} where the  complete rigidity corollaries are carried out.

\medskip

{\em{Acknowledgements.}}
We would like to thank M.-P. Aparicio-Gomez, T. Haettel, J. Heitsch, H. Oyono-Oyono, S. Paycha, J. Renault, I. Roy and G. Skandalis for many useful discussions. We are especially indebted to P. Carrillo-Rouse and  P. Piazza  for their precious comments on a preliminary version of this work.
{{Finally, the authors would like to thank the referee for her/his careful reading of the manuscript.}}

\medskip

\tableofcontents

\medskip

%
%
%
%
%
%
%
%
%

\vspace{1\baselineskip}

\section{Preliminary results}\label{prelim}

\subsection{Lifts of operators on Hilbert modules}

We shall use the standard definition of Hilbert modules and adjointable operators between them. See for instance \cite{Lance}. 

\medskip
%
%


Given Hilbert $A$-modules $E$ and $E'$, the space of adjointable operators from $E$ to $E'$ will be abusively denoted $\maL_A(E, E')$. If $(x, x')\in E\times E'$, then we denote by  $\theta_{x', x}\in \maL_A(E, E')$ the  ``one-dimensional'' adjointable operator   defined by
$$
\theta_{x', x} (z) = x' \; \langle x, z\rangle, \quad\text{ for }z\in E.
$$
The subspace composed of  finite sums of such one-dimensional operators is the space of ``finite-dimensional'' adjointable operators, its closure in $\maL_A(E, E')$  is the space $\maK_A(E, E')$ of $A$-compact operators, or simply compact operators when no confusion can occur. Notice that if $E'=E$ then $\maK_A(E, E)=\maK_A(E)$ is a closed two-sided involutive ideal in $\maL_A(E)$, exactly as for the usual case of $A=\C$. In particular $\maL_A(E)$ is the multiplier algebra of $\maK_A(E)$. 

\begin{example}\label{AH}\
For any separable Hilbert space  $H$,  $H\otimes A$,  is a Hilbert $A$-module. Then $\maK_A(A\otimes H)$ is isomorphic to the $C^*$-algebra  $A\otimes \maK(H)$. 
 If we identify $H$ with $ \ell^2(\N)$ using an orthonormal basis, then $H\otimes A$ is identified with  the space $\ell^2(\N, A)$ of $A$-valued sequences $(a_n)_{n\in \N}$ which are square summable, i.e. such that the series $\sum_{n\in \N} a_n^* a_n$ converges in $A$. 
 \end{example}

Let $E$ be a Hilbert module over a $C^*$-algebra $B$ and $F$ a Hilbert module over a $C^*$-algebra $A$. Assume  that $\alpha:B\to\maL_A(F)$ is a homomorphism of $C^*$-algebras, also called a representation of $B$ in the Hilbert $A$-module $F$. Then  we denote by $E\otimes_\alpha F$ the composition Hilbert $A$-module,  see for instance  \cite{Lance}. 

For any element $e\in E$ we then denote by $L_e\in\maL_A(F,E\otimes_BF)$ the operator defined by $L_e(f)=e\otimes f$ for $f\in F$. Then, $L_e$ is adjointable with the adjoint  given by $(L_e)^*(e'\otimes f)=\alpha (\langle e,e'\rangle)f$  for any $e'\in E$ and $f\in F$.

\medskip

\begin{definition}\cite{HigsonRoe, ConnesSkandalis}\
\label{lift}
Let $T\in\maL_A(F)$ be an adjointable operator. An operator $T_E\in\maL_A(E\otimes_\alpha F)$ will be called  a \emph{lift} of $T$ if for any $e\in E$,  the following diagrams commute up to  compact operators:
$$\xymatrix{
{E\otimes_\alpha F} \ar@{->}[r]^{T_E} \ar@{<-}[d]_{L_e}  & {E\otimes_\alpha F} \ar@{<-}[d]^{L_e} 
&\mbox{ and } &
{E\otimes_\alpha F} \ar@{->}[r]^{T_E} \ar@{->}[d]_{(L_e)^*}  & {E\otimes_\alpha F} \ar@{->}[d]^{(L_e)^*}\\
{F} \ar@{->}[r]^{T}& {F} & & {F} \ar@{->}[r]^{T}& {F} 
}$$
\end{definition}

\medskip

Recall that a Hilbert module is finitely generated projective if the identity is a compact operator. The following result is proved in \cite{ConnesSkandalis}, Appendix A (where the term \textit{connection} is used).


\medskip

\begin{lemma}\cite{ConnesSkandalis}\label{CS}\
If $A$ and $B$ are $\sigma$-unital $C^*$-algebras and $E$ is a finitely generated projective Hilbert $B$-module, then every operator $T\in\maL_A(F)$ which commutes up to compacts with the representation $\alpha$ of $B$ admits a lift $T_E\in\maL_A(E\otimes_\alpha F)$. 
\end{lemma}

The expression ``$T$ commutes up to compacts with the representation $\alpha$ of  $B$'' means that the commutators $[T, \alpha(b)]$ are $A$-compact operators of $F$, for all $b\in B$.

\medskip


\medskip

Notice that if $T_E$ is a lift of $T$, then  for any $e,e'\in E$, the operator $
(L_e)^*T_EL_{e'}-\alpha(\langle e,e'\rangle)T$ is a compact operator of the Hilbert $A$-module $F$.
Indeed  $\alpha(\langle e,e'\rangle)=(L_e)^*L_{e'}$ and hence
$$
(L_e)^*T_EL_{e'}-\alpha (\langle e,e'\rangle)T=(L_e)^*(T_E L_{e'} -L_{e'} T)\text{ is compact}.
$$
We also need  the following obvious converse when $E$ is a finitely generated projective module.

\medskip

\begin{proposition}
Assume that  $B$ is unital and that $E$ is a finitely generated projective Hilbert $B$-module. Let $T\in\maL_A(F)$ be as before an adjointable operator which commutes up to compacts with the representation $\alpha$ of $B$, and let $T_E\in \maL_A(E\otimes_\alpha F)$ be a given adjointable operator. Then $T_E$ is a lift of $T$ if and only if 
$$
(L_e)^*T_EL_{e'}-\alpha(\langle e, e'\rangle)T\, \in\, \maK_A(F), \quad \forall e,e'\in E.
$$
\end{proposition}

\begin{proof}
There exists a projection $P\in\maL_B(B^n)$ such that $E\simeq P(B^n)$.
Set $e_i=P\ep_i$ where $(\ep_1,\dots,\ep_n)$ is the canonical basis of $B^n$, we compute  for any $e\in E$, $e=(b_1,\dots b_n)$, and any $f\in F$ 
\begin{eqnarray*}
L_{e_i}(L_{e_i})^*(e\otimes f) 
&=& P(\ep_ib_i)\otimes f.
\end{eqnarray*}
Hence 
$
\sum_{i=1}^nL_{e_i}(L_{e_i})^*=\id_{E\otimes_\alpha  F}$, and therefore
\begin{eqnarray*}
T_EL_e - L_e T 
&=&\sum_{i=1}^nL_{e_i}((L_{e_i})^*T_EL_e-\alpha (\langle e_i,e\rangle)T ).
\end{eqnarray*}
The similar relation
$$
((L_{e'})^*T_E^*L_e-(L_{e'})^*L_eT^*)^*=((L_e)^* T_E-T (L_e)^*) L_{e'},
$$
allows to prove similarly that   $(L_e)^*T_E-T(L_e)^*$ is  compact. 
\end{proof}

\begin{remark}\ 
The previous proposition is still valid for non unital $B$ when $E=pB^n$  with $p$ a projection in $M_n(B)$. Indeed, in this case $pB^n=p{\widetilde B}^n$ where ${\widetilde B}$ is the minimal unitalization of $B$. 
\end{remark}

%

%

\medskip

\subsection{The regular HR sequence}

First we fix the notations and rules for our transformation groupoid $\maG$. Let $\Gamma$ be a countable discrete group and let  $X$ be a compact Hausdorff  space together with a given action of $\Gamma$ by homemorphisms on the left.  We shall simply denote by $\gamma x$ the result of the action of the homemorphism of $X$ associated with $\gamma\in \Gamma$ on the chosen element $x\in X$. The neutral element of the group $\Gamma$ will be denoted $1_\Gamma$, and for $\gamma\in \Gamma$, $\delta_\gamma$ will usually denote the characteristic function of $\{\gamma\}$ in $\Gamma$. 
We consider the crossed product groupoid $\maG=X\rtimes \Gamma$ defined as follows. The space of units is $X$, and the space of arrows is the cartesian product $X\times\Gamma$ with its topology inherited from $X$. The operation rules are given by
$$
(\gamma x,\gamma)(x,\mu)=(\gamma x,\gamma \mu)\text{ so }s(x,\gamma)=\gamma^{-1} x \mbox{  and  } r(x,\gamma)=x,
$$
where $s$ and $r$ are the source and range map $\maG$.

%

Let $Z$ be a locally compact proper metric space, and set $Y=Z\times X$. We denote by $\rho: Y\to X$ the second projection, so that $Y$ can be understood as a specific $\maG$-space with the anchor map being given by $\rho$. Then $C_0(Y)$ is a $C(X)$-algebra when we set
$$
gf=fg := {{(\rho^*f)  g}}\text{ for }f\in C(X)\text{ and } g\in C_0(Y).
$$
Given a Hilbert $C(X)$-module $E$, it is easy to check that the $C^*$-algebra $\maL_{C(X)}(E)$ is also a $C(X)$-algebra, see for instance \cite{BenameurRoyI}. 
A specific $C(X)$-Hilbert module is given by $C(X)\otimes H$ for a given Hilbert space $H$, as explained in Example \ref{AH}. Then an operator $T\in\maL_{C(X)}(C(X)\otimes H)$ is canonically  given by a $*$-strongly continuous  field $(T_x)_{x\in X}$ of bounded operators on $H$.

For any   $\maG$-space $Y$,  a \underline{$C(X)$-representation} of the $C(X)$-algebra $C_0(Y)$ is  a pair $(E, \pi)$ where $E$ is a Hilbert $C(X)$-module, and $\pi: C_0(Y)\to\maL_{C(X)}(E)$  is a $C(X)$-homomorphism, i.e. a $*$-homomorphism such that
$$
\pi(f g)e=\pi(g)(ef), \quad \text{ for }f\in C(X), g\in C_0(Y)\text{ and }  e\in E,\: .
$$
In particular, a $C(X)$-representation $\pi: C_0(Y)\to\maL_{C(X)}(C(X)\otimes H)$ assigns to any $f\in C_0(Y)$, the operator $\pi(f)$ which can be written as the $*$-strongly continuous field $(\pi_x(f))_{x\in X}$ where each $\pi_x:C_0(Y)\to\maL(H)$ is a $^*$-representation of $C_0(Y)$  in the Hilbert space $H$,  which factors through $C_0(Y_x)$ so that writing $f_x=f\vert_{Y_x}$, $\pi_x(f)=\pi_x(f_x)$. {{Here and in the sequel $Y_x$ denotes the fiber of the anchor map $Y\to X$}}. When $Y=Z\times X$ as above, $C_0(Z)$ embeds as a $C^*$-subalgebra of $C_0(Y)$ and 
$\pi$ is thus given by the family $(\pi_x)_x$ of representations of $C_0(Z)$. We shall only consider this latter case in the sequel.

 
\begin{definition}\
\begin{itemize}
\item[(i)] {{The $C(X)$-representation  $(E, \pi)$   is fiberwise \underline{faithful} if for any $x\in X$, the representation $\pi_x$ in $E_x$ is faithful.}}
\item[(ii)] {{The $C(X)$-representation  $(E, \pi)$   is fiberwise \underline{non-degenerate} if for any $x\in X$, the representation $\pi_x$ in $E_x$ is non-degenerate, i.e. $\pi_x (C_0(Z))E_x$ is dense in $E_x$.}}
\item[(iii)] {{The $C(X)$-representation  $(E, \pi)$ is fiberwise \underline{standard} if for any $x\in X$, the representation $\pi_x$ in $E_x$ is standard, i.e. $\pi_x (g)\in \maK(E_x)\Rightarrow g=0$, for $g\in C_0(Z)$.
\item[(iv)] The $C(X)$-representation  $(E, \pi)$ is fiberwise  \underline{ample} if it is  fiberwise non-degenerate and  fiberwise standard.}}
\end{itemize}
\end{definition}

\medskip

{{For simplicity, we shall sometimes drop the word ''fiberwise'', so for instance an ample representation will mean  fiberwise ample representation, and similarly for faithful representation which means in the present paper fiberwise faithful representation. Said differently, $(E, \pi)$ is ample if for any $x\in X$, the representation $\pi_x$ is ample in the usual sense.  Notice that if $\pi(C_0(Y)) (E)$ is dense in $E$, then the representation is  non-degenerate. Also, if $\pi$ is standard, then $\pi (\varphi) \in \maK_{C(X)} (E) \Rightarrow \varphi=0$ for any $\varphi\in C_0(Y)$, so that one can be tempted to use this other definition. Our choice is motivated by the notion of homogeneous $X$-extension as introduced and studied in \cite{PPV1, PPV2}, see  \cite{BenameurRoyII} for more details. When working with a constant field of representation, say with a fixed  representation of $C_0(Z)$ in a fixed Hilbert space $H$, these subtelties disappear. In the sequel, and unless otherwise specified, our representations will always be non-degenerate. It is easy to check that for any given non-degenerate representation in $E$, {{one obtains}} an ample representation by tensoring with the identity of the Hilbert space $\ell^2(\N)$.}}

Let us fix  a separable (infinite dimensional) Hilbert space $H$ and  a non-degenerate $C(X)$-representation $\pi: C_0(Y)\to\maL_{C(X)}(C(X)\otimes H)$. Recall that  $\Gamma$ acts properly on the right  on the locally compact metric space $Z$. We assume that the action on $Z$ is cocompact and that the metric of $Z$ is a proper metric which is $\Gamma$-invariant.

\medskip

\begin{definition} 
An operator $T\in\maL_{C(X)}(C(X)\otimes H)$ has finite propagation (with respect to $\pi$) if there exists $R>0$ such that $
\pi (\varphi) T\pi (\psi) = 0$ for any $\varphi, \psi\in C_c (Z)$  with $d_Z (\Supp (\varphi), \Supp (\psi)) >R$. 
The smallest such constant $R$ is the propagation of $T$. 
%
%
\end{definition}

\medskip

Assume  from now on  
that $H$ is a Hilbert space unitary representation of $\Gamma$, so we also have a unitary representation $U:\Gamma\to\maU(H)$. We shall sometimes call such $H$ a $\Gamma$-Hilbert space.
An operator $T\in\maL_{C(X)}(C(X)\otimes H)$ is then \emph{$\Gamma$-equivariant}  if the associated field $(T_x)_{x\in X}$ is a $\Gamma$-equivariant field over $X$, i.e. it satifies
$$
T_{\gamma x}=U_\gamma T_xU_\gamma^*, \quad \forall x\in X,\:\forall\gamma\in\Gamma,\: 
$$
The subspace of $\Gamma$-equivariant adjointable operators will be denoted  $\maL_{C(X)}(C(X)\otimes H)^\Gamma$.

If the operator $T$ commutes with the representation $\pi$ up to compact operators, say 
$$
T\pi(g)-\pi(g)T\in \maK_{C(X)} (C(X)\otimes H),\quad \forall g\in C_0(Z\times X),
$$
then the operator $T$ will be called a \emph{pseudolocal} operator.
When $\pi(g)T$ and $T\pi(g)$ are already compact operators  for any $g\in C_0(Z\times X)$, then $T$ will be called a \emph{locally compact} operator.

\begin{definition}\cite{BenameurRoyII}\label{EquivariantRoeAlgebras}
Recall that $Z$ is a proper cocompact locally compact Hausdorff $\Gamma$-space, and $H$ is a $\Gamma$-equivariant Hilbert space which is endowed with the $\Gamma$-equivariant $C(X)$-representation $\pi: C_0(Z\times X)\to \maL{_{C(X)}}({C(X)\otimes} H)$.  
\begin{enumerate}
\item[•] The dual Roe algebra associated with our data, denoted $D^*_\Gamma(X; Z,H)$,  will be the norm closure  in $\maL_{C(X)}(C(X)\otimes H)$ of the subspace composed of $\Gamma$-equivariant finite propagation pseudolocal operators.  
\item[•] The dual Roe ideal associated with our data, denoted $C^*_\Gamma(X; Z,H)$, will be the norm closure in $\maL_{C(X)}(C(X)\otimes H)$ of the subspace composed of $\Gamma$-equivariant finite propagation locally compact operators.  
\end{enumerate}
\end{definition}

Clearly every element of $D^*_\Gamma(X; Z,H)$ is a pseudolocal $\Gamma$-equivariant operator, and every element of $C^*_\Gamma(X; Z,H)$ is a locally compact $\Gamma$-equivariant operator. 
%

\begin{lemma}\cite{BenameurRoyI}\
The $C^*$-algebra $C^*_\Gamma(X; Z,H)$ is a two-sided  involutive ideal in the $C^*$-algebra  $D^*_\Gamma(X; Z,H)$. Moreover,   an operator $T\in D^*_\Gamma(X; Z,H)$ belongs to the ideal $C^*_\Gamma(X; Z,H)$ if and only if $\pi(g)T$ is a $C(X)$-compact operator for any 
$g\in C_0(Z\times X)$.
\end{lemma}

We shall denote by $Q^*_\Gamma(X; Z,H)$ the quotient $C^*$-algebra, so that we have the following short exact sequence of $C^*$-algebras
$$
0\to C^*_\Gamma(X; Z,H)\hookrightarrow D^*_\Gamma(X; Z,H)\to Q^*_\Gamma(X; Z,H)\to 0.
$$

Applying the topological $K$-functor, we deduce well defined boundary maps
$$
\partial_0 \;:\; K_0\left(Q^*_\Gamma (X; Z,H)\right) \longrightarrow K_{1}\left( C^*_\Gamma (X; Z, H)\right)  \text{ and } \partial_1 \;:\; K_1\left(Q^*_\Gamma (X; Z,H)\right) \longrightarrow K_{0}\left( C^*_\Gamma (X; Z, H)\right).
$$
which fit in  the corresponding  periodic six-term $K$-theory exact sequence.

Following again \cite{BenameurRoyII}, the Paschke-Higson maps  can be described for our specific groupoid $\maG=X\rtimes \Gamma$ and yield group morphisms
$$
\maP_0^{Z,H} \;:\;  K_0\left( Q^*_\Gamma (X; Z,H)\right) \longrightarrow KK^{1}_\Gamma ( Z, X)\text{ and }\maP_1^{Z, H} \;:\;  K_1\left( Q^*_\Gamma (X; Z,H)\right) \longrightarrow KK^{0}_\Gamma ( Z, X)
$$
where $KK^{*}_\Gamma ( Z, X)$ is the $\Gamma$-equivariant Kasparov $KK$-theory of the pair of $\Gamma$-algebras $C_0(Z), C(X)$, see \cite{Kasparov}. Moreover, at the cost of replacing if necessary the  representation $H$ by the amplification ${\what H}:= H\otimes \ell^2\Gamma\otimes \ell^2\N$ with the representation $\pi\otimes \id\otimes \id$ {{where the unitary representation $U$ of $\Gamma$ is replaced  by $U\otimes\rho\otimes \id$ with $\rho$ being the regular representaion, it was proved in \cite{BenameurRoyI} and \cite{BenameurRoyII} that we can insure the following properties at once:}}
\begin{enumerate}
\item The $K$-theory groups of the three Roe algebras are {{independent}} of the choice of the non-degenerate representation $\pi$ in $C(X)\otimes H$;
\item The Paschke-Higson maps $\maP_i^{Z,{\what H}}$ are isomorphisms;
\item the $C^*$-algebra $C^*_\Gamma (X; Z, {\what H})$ is Morita equivalent to the reduced $C^*$-algebra $C^*_r (\maG)$;  and
\item the boundary maps $\pa_i$ coincide, via the above two isomorphisms,  with the (reduced) Baum-Connes maps for $\Gamma$ with coefficients in $C(X)$, or equivalently with  the Baum-Connes maps for our transformation groupoid $\maG$. 
\end{enumerate}

We thus end up with the six-term exact sequence
$$
\begin{diagram}
       \node{K_0(C^*_r\maG)}\arrow{e}\node{K_0 (D^*_\Gamma (X; Z))}\arrow{e}\node{KK_\Gamma^1(Z, X)}\arrow{s}\\
\node{KK_\Gamma^0(Z, X)}\arrow{n}\node{K_1(D^*_\Gamma (X; Z))}\arrow{w} \node{K_1(C^*_r\maG)}\arrow{w}
\end{diagram} 
$$
with vertical maps given by the Baum-Connes maps associated with  the proper cocompact $\Gamma$-space $Z$, with coefficients in $C(X)$. We have removed $\what{H}$ from the notation since the isomorphism class of the group $K_i (D^*_\Gamma (X; Z, \what{H}))$ does not depend of $\pi$. 

In \cite{BenameurRoyII}, the functoriality of these construction with respect to inclusion of closed subsets was proved, giving  the universal sequence
$$
 \begin{diagram}
      \node{K_0(C^*_r\maG)}\arrow{e}\node{\maS^r_1(\maG)}\arrow{e}\node{\maR K_1(\maG)}\arrow{s}\\
\node{\maR K^0(\maG)}\arrow{n}\node{\maS^r_0(\maG)}\arrow{w} \node{K_1(C^*_r\maG)}\arrow{w}
\end{diagram}
$$
Here $\maR K^i(\maG)$ is the LHS of the Baum-Connes assembly map for the groupoid $\maG$, while $\maS^r_i(\maG)$ are universal reduced structure groups associated with the {{\'etale}} groupoid $\maG=X\rtimes \Gamma$ and constructed as inductive limits exactly as $\maR K^i(\maG)$, see for instance \cite{BaumConnesHigson} or the review given in \cite{BenameurRoyII}. 

It is the aim of the next section to set up a similar six-term exact sequence involving other admissible crossed product completions, especially the maximal one. 


\subsection{Admissible completions}

The space  $C_c(\maG)$ of compactly supported continuous functions, on the space of arrows, is an involutive convolution algebra for the following rules

$$(f\ast g)(x,\gamma)=\sum_{\mu\in\Gamma}f(x,\mu)g(\mu^{-1} x,\mu^{-1}\gamma)
\mbox{  and  }
f^*(x,\gamma)=\overline{f(\gamma^{-1} x,\gamma^{-1})}, \quad \text{ for } f, g\in C_c(\maG).
$$
Let as usual  $L^1(\maG)$ be the completion of $C_c(\maG)$ with respect to the Banach norm $\Vert\bullet \Vert_1$ defined by
$$
\Vert f\Vert_1 := \max\left(\sup_{x\in X}\left(\sum_{\gamma\in\Gamma}\vert f(x,\gamma)\vert\right), \sup_{x\in X}\left(\sum_{\gamma\in\Gamma}\vert f(\gamma^{-1} x,\gamma^{-1})\vert\right)\right).
$$
Then the maximal norm  $\Vert \bullet \Vert_{\max}$ on $L^1(\maG)$ is defined as
$$
\Vert f\Vert_{\max} := \sup\lbrace \Vert\pi(f)\Vert,\: \pi: L^1(\maG)\rightarrow\maB(H) \mbox{ a continuous involutive representation}\rbrace.
$$
The (maximal or full) $C^*$-algebra of the groupoid $\maG$ is the completion of $L^1(\maG)$ with respect to the maximal norm. We denote it by $C_{\max}^*(\maG)$. 


An important  continuous and faithful representation of $L^1(\maG)$ is the regular left  representation $\lambda$ defined as follows, see \cite{RenaultBook}.
For  $x\in X$, let $\lambda_x: C_c(\maG)\rightarrow B(\ell^2\Gamma)$ be the representation at $x$ given by 
$$
\lambda_x(f)(\xi)(\gamma)=\sum_{\mu\in\Gamma}f(\gamma x,\gamma\mu^{-1})\xi(\mu).
$$
Rather completing $L^1(\maG)$ with respect to the resulting regular norm $
\Vert \bullet\Vert := \sup_{x\in X}\Vert \lambda_x(\bullet)\Vert$,
we  obtain the regular, or reduced, $C^*$-algebra associated with the \'etale groupoid $\maG$, it will be denoted  $C^*_r(\maG)$. We thus have by definition of the maximal $C^*$-algebra, a well defined epimorphism  extending $\lambda$:
$$
\lambda: C_{\max}^*(\maG)\longrightarrow C^*_r(\maG).
$$
According to \cite{BaumGuentnerWillett}, we may as well use any completion $C^* (\maG)$ between the previous two completions, so such that we have morphisms
$$
C_{\max}^*(\maG) \longrightarrow C^* (\maG)\longrightarrow C^*_r(\maG).
$$
So with respect to  a $C^*$-norm $\vert\vert \bullet\vert\vert$ on $C_c(\maG)$ which satisfies $
\vert\vert \bullet\vert\vert_r\leq \vert\vert \bullet\vert\vert \leq \vert\vert \bullet\vert\vert_{\max}$.

\vspace{1\baselineskip}

Notice that the previous $C^*$-algebras of $\maG$ are easily identified with crossed product completions for the action of $\Gamma$ on $C(X)$. More generally, we shall also consider other admissible crossed products, and also  strongly continuous actions of $\Gamma$ on some $C^*$-algebras. A  $\Gamma$-algebra $A$ is a $C^*$-algebra  endowed with a (left) action of our discrete countable group $\Gamma$, the precise action being implicit. 
 Given a unitary $\Gamma$-Hilbert space $H$, say a Hilbert space with a unitary representation $U : \Gamma\to\maU(H)$ of $\Gamma$,  and an involutive representation $\pi :A\to\maL(H)$, we shall call the pair $(\pi, U)$ a covariant representation of $(A, \Gamma)$ if the following covariance relations hold
$$
\pi(\gamma a)=U_\gamma\pi(a)U_\gamma^*,  \quad \forall (a, \gamma) \in A \times \Gamma.
$$
Given such a covariant representation, the space $C_c(\Gamma, A)$ of finitely supported functions from $\Gamma$ to $A$, has the structure of an involutive convolution algebra given as follows:
$$
(f*g)(\gamma)=\sum_{\mu\in\Gamma}f(\mu)\, \mu g(\mu^{-1}\gamma) \mbox{  and  }
f^*(\gamma)=\gamma f(\gamma^{-1})^*, \quad \text{ for }f, g\in C_c(\Gamma, A)\text{ and }\gamma\in \Gamma.
$$
We then obtain a well defined involutive representation  of the algebra $C_c(\Gamma, A)$ that we shall denote by $\pi\rtimes U:C_c(\Gamma, A) \rightarrow \maL (H)$, and which is defined by (see for instance \cite{Williams}[Proposition 2.23]): 
$$
(\pi\rtimes U)(f)=\sum_{\mu\in\Gamma}\pi(f(\gamma))\circ U_\gamma, \quad  \text{ for  } f\in C_c(\Gamma,A).
$$
 The maximal completion with respect to the system of all such covariant representations is the maximal (also called  full) crossed product $C^*$-algebra, it will be denoted $A\rtimes_{\max}\Gamma$. So for $f\in C_c(\Gamma, A)$, the formula
$$
\Vert f\Vert_{\max} := \sup\lbrace\Vert(\pi\rtimes U)(f)\Vert,\:(\pi,U)
\mbox{ a covariant representation of } (A,\Gamma)\rbrace,
$$
defines a $C^*$-norm, see for instance \cite{Williams}[Lemma 2.27]. Then $A\rtimes_{\max}\Gamma$ is just the completion of $C_c(\Gamma, A)$ with respect to this norm.

As in \cite{BaumGuentnerWillett}, we may consider any admissible crossed product. Recall that a  crossed product $\rtimes \Gamma$ is a functor from the category of $\Gamma$-algebras to that of $C^*$-algebras satisfying some properties. In particular, we shall only consider those crossed products which for any $\Gamma$-algebra $A$, yield a $C^*$-algebra $A\rtimes \Gamma$ which sits between the maximal crossed product $A\rtimes_{\max} \Gamma$ and the reduced crossed product $A\rtimes_r\Gamma$, so with natural epimorphic transformations 
$$
A\rtimes_{\max} \Gamma \longrightarrow A\rtimes \Gamma \longrightarrow A\rtimes_r \Gamma.
$$
For examples of exotic crossed products $\rtimes$, meaning strictly sitting between the reduced and the maximal ones, we refer the reader for instance to \cite{BaumGuentnerWillett}, see also \cite{EchterhoffWillett}. Considering any admissible space $\maR$ of unitary equivalence classes of such representations, one may introduce interesting examples of exotic crossed products $A\rtimes \Gamma$, see \cite{BrownGuentner}. 

The crossed product functor $\rtimes\Gamma$ is compatible with Morita equivalence if for any separable $\Gamma$-algebra $A$, the well defined maximal untwisting isomorphism 
$$
[A\otimes \maK (\ell^2(\Gamma)^\infty)]\rtimes_{\max} \Gamma \longrightarrow (A\rtimes_{\max}\Gamma)\otimes \maK (\ell^2(\Gamma)^\infty),
$$
descends to an isomorphism for $\rtimes\Gamma$, see again \cite{BaumGuentnerWillett}[Definition 3.2].  Thus, such crossed product functor transforms $\Gamma$-Morita equivalent  separable $\Gamma$-algebras to Morita equivalent $C^*$-algebras. Indeed, if $A$ and $B$ are Morita equivalent separable $\Gamma$-algebras, then equivariant Morita equivalence always yields an isomorphism
$$
[A\otimes \maK (\ell^2(\Gamma)^\infty)]\rtimes  \Gamma \simeq [B\otimes \maK (\ell^2(\Gamma)^\infty)]\rtimes  \Gamma.
$$
Therefore
$$
(A\rtimes\Gamma)\otimes \maK (\ell^2(\Gamma)^\infty) \simeq [A\otimes \maK (\ell^2(\Gamma)^\infty)]\rtimes \Gamma \simeq [B\otimes \maK (\ell^2(\Gamma)^\infty)]\rtimes  \Gamma \simeq (B\rtimes\Gamma)\otimes \maK (\ell^2(\Gamma)^\infty).
$$
and we see that $A\rtimes\Gamma$ and $B\rtimes\Gamma$ are Morita equivalent. Notice that  the regular  crossed product is Morita compatible, see also  \cite{BaumGuentnerWillett}[Appendix A]. 
Exact crossed product have better behaviour with respect to the Baum-Connes assembly map, so our constructions are especially interesting for them. It was  proved in \cite{BaumGuentnerWillett} that there always  exists a minimal exact admissible crossed product $\rtimes_{\ep} \Gamma$, and which coincides with the regular one precisely when the group is exact. Recall that exactness  means that for any short exact sequence of $\Gamma$-algebras
$$
0\to I\hookrightarrow A \rightarrow A/I \to 0,
$$
the corresponding sequence of $C^*$-algebras
$$
0\to I\rtimes\Gamma \hookrightarrow A\rtimes\Gamma \rightarrow (A/I)\rtimes\Gamma \to 0,
$$
is also exact. The maximal crossed product is for instance exact, while the reduced is not. In fact the reduced crossed product is exact if and only if the  reduced $C^*$-algebra $C^*_r(\Gamma)$ is exact \cite{KirchbergWassermann}. 

\begin{remark}
As explained in the introduction, the privileged crossed product for the rigidity applications of rho invariants will be the maximal crossed product. {{Therefore, the}} reader can suppose in first reading and for simplicity that all the admissible crossed products considered in the sequel coincide with  the maximal one. 
\end{remark}

\vspace{1\baselineskip}

\subsection{Generalized Connes-Skandalis module}

Let $Z$ be a proper cocompact right $\Gamma$-space as before. We fix  as before a non-degenerate $C(X)$-representation $\pi: C_0(Z\times X)\to\maL_{C(X)}(C(X)\otimes H)$. Notice that for  $g\in C_0(Z\times X)$, the operator $\pi(g)$ can be identified with a  $^*$-strongly continuous field $(\pi_x(g))_{x\in X}$ and moreover each $\pi_x:C_0(Z\times X)\to\maL(H)$ is determined by its restriction to $C_0(Z)$, still denoted $\pi_x:C_0(Z)\to \maL (H)$. 
We also fix a unitary representation $U :\Gamma\to\maU(H)$, and we suppose that $\pi$ is $\Gamma$-equivariant, that is for all $g\in C_0(Z)$,
$$
\pi_{\gamma x}(g)=U_\gamma\pi_x(\gamma^{-1} g)U_{\gamma^{-1}}\quad  \text{ where } \quad (\gamma g)(z)=g(z\gamma).
$$

We proceed now to associate with these data  a  Hilbert module $\maE^H$ over the $C^*$-algebra $C^*(\maG)$ which is the natural generalisation of the Connes-Skandalis Hilbert module, see for instance \cite{BenameurPiazza}. We first define the Hilbert module $\maE^H$ as the completion of the pre-Hilbert module over the algebra $C_c(X\times\Gamma)$ defined by  $\maE_c^H:=\pi(C_c(Z\times X))(C(X)\otimes H)$. 

\begin{lemma}
The space $\maE^H_c$ is a right pre-Hilbert module over  the convolution algebra $C_c(X\times\Gamma)$ for the rules
$$
\quad {{(\xi_1 f)(x)=\sum_{\gamma\in\Gamma}f(\gamma x,\gamma)U_{\gamma^{-1}}\xi_1(\gamma x)\; \text{ and }\;
\langle\xi_1,\eta_1\rangle (x,\gamma)=\langle \xi_1(x),U_\gamma\eta_1(\gamma^{-1} x)\rangle}},
$$
where {{$\xi_1, \eta_1\in\maE_c^H$}} and $f\in C_c(X\times\Gamma)$.
\end{lemma}

\begin{proof}\ 
We have for  $g\in C_c(Z\times X)$ and {{$\xi\in C(X)\otimes H$}} with the usual notation $(\gamma g)(z,x)=g(z \gamma,\gamma^{-1} x)$:
$$
{{[\pi(g)(\xi)] f \,}}(x) = \sum_{\gamma\in\Gamma}f(\gamma x,\gamma)U_{\gamma^{-1}}\pi_{\gamma x}(g_{\gamma x})\xi(\gamma x)
$$
But since $\pi$ is $\Gamma$-equivariant, we have $U_{\gamma^{-1}}\circ \pi_{\gamma x}(g_{\gamma x})= \pi_x(\gamma^{-1} g_{\gamma x})\circ U_{\gamma^{-1}}$. Hence we get
$$
{{[\pi(g)(\xi)] f \,}}(x) = \sum_{\gamma\in\Gamma}\pi_x((\gamma^{-1} g)_x)f(\gamma x,\gamma)U_{\gamma^{-1}}\xi(\gamma x) {{=\pi (g) (\xi f) (x)}}.
$$
Hence,  {{with $\xi_1= \pi(g)(\xi)$, we deduce that $\xi_1 f \in\maE_c^H$, and finally we conclude that $\xi_1 f\in \maE_c^H$ for any $\xi_1\in\maE_c^H$}}. Moreover we can compute
\begin{eqnarray*}
\langle\pi(g)\xi,\pi(h)\eta\rangle (x,\gamma) &=& \langle\pi_x(g_x)\xi(x),U_\gamma\pi_{\gamma^{-1} x}(h_{\gamma^{-1} x})\eta(\gamma^{-1} x)\rangle\\
&=& \langle \xi(x),\pi_x({\overline{g_x}} (\gamma h_{\gamma^{-1} x}))U_\gamma\eta(\gamma^{-1}x)\rangle.
\end{eqnarray*}
Hence, $\langle \pi(g)\xi,\pi(h)\eta\rangle \in C_c(X\times\Gamma)$ since $g$ and $h$ are compactly supported and since $\Gamma$ acts properly on the locally compact space $Z$.

That these operations endow $\maE_c^H$ with the structure of a pre-Hilbert module over the involutive algebra $C_c(X\times\Gamma)$ is then a straightforward verification, see for instance \cite{VictorThesis}.

\end{proof}

\begin{definition}[The Connes-Skandalis module]\label{CS-HilbertModule}\ The completion of the  pre-Hilbert  $C_c(X\times\Gamma)$-module $\maE_c^H$  with respect to the $C^*$-norm of $C^*(\maG)$ yields a right Hilbert $C^*$-module over the $C^*$-algebra $C^*(\maG)\simeq C(X)\rtimes \Gamma$. It will be denoted  $\maE^H$ and sometimes called the  Connes-Skandalis Hilbert module associated with $H$ and with the fixed completion $C^*$-algebra $C^*(\maG)$. 
\end{definition}
\medskip

When the completion is taken with respect to the maximal crossed product, 
we denote by $\maE_{\max}^H$ the resulting maximal Connes-Skandalis Hilbert module over $C^*_{\max} (\maG)$.  In the same way, $\maE^H_{\ep}$ will be the Connes-Skandalis Hilbert module over the minimal exact  admissible crossed product $C^*$-algebra $C^*_{\ep} (\maG)=C(X)\rtimes_{\ep}\Gamma$ introduced in \cite{BaumGuentnerWillett}. Finally, the regular Connes-Skandalis Hilbert module $\maE_r^H$ is the module over the $C^*$-algebra $C(X)\rtimes_r \Gamma$ obtained using the regular completion. 

\begin{example}
The terminology {\em{Connes-Skandalis}} Hilbert module is motivated by the special case $H=L^2(Z)$ for some (fully supported if we impose faithfulness) $\Gamma$-invariant Borel measure on $Z$ where we  recover the usual definition of   Connes-Skandalis Hilbert module as introduced in \cite{BenameurPiazza}, following the more general definition given for foliation groupoids in \cite{ConnesSkandalis} in the reduced case.  A similar construction was also carried out in \cite{Renault87}, see in particular Corollary 5.2 there. 
\end{example}

\color{black}

\begin{example}\label{SmoothCSmodule}\
Assume that $Z:=\tM\to M$ is some Galois covering over the smooth closed manifold $M$ with group $\Gamma$, and let $\tE\to \tM$ be the lift to $\tM$ of some hermitian vector bundle over the quotient smooth compact foliated manifold $M$.  Then using a $\Gamma$-invariant  Lebesgue measure on $\tM$, and the corresponding Hilbert space $H=L^2(\tM, \tE)$ of classes of $L^2$-sections of $\tE$ over $\tM$,  we get  the standard Connes-Skandalis Hilbert module \cite{BenameurPiazza}. 
\end{example}

 The above construction of the Hilbert module $\maE^H$ can be generalized in many directions. For instance, one may consider instead of the free Hilbert module $C(X)\otimes H$, any  $\Gamma$-equivariant continuous field of Hilbert spaces over $X$ endowed with a $\Gamma$-equivariant representation of $C_0(Z\times X)$. It is also a routine exercise to extend the previous definitions to the case of more general $\maG$-proper and cocompact spaces $(Y, \rho:Y\to X)$ where $\rho$ is the $\Gamma$-equivariant anchor map generalizing the second projection $Z\times X\to X$. In the prototype example of Galois coverings of smooth manifolds, this means working with more general closed foliated manifolds and more general hermitian vector bundles. These extensions would though introduce some extra-technicalities and are not needed for the purpose of the present  paper.

In the rest of this section, we give other descriptions of the Hilbert module $\maE^H$ which will also be needed in the next sections.

Recall first  the crossed product Hilbert module $(C(X)\otimes H)\rtimes \Gamma$ over $C^*(\maG)$, which is the completion of the $C_c(X\times\Gamma)$ pre-Hilbert module $C_c(X\times\Gamma, H)$ for the fixed admissible $C^*$-norm. More precisely, the right module structure is given by the formula
$$
(\xi f)(x, \gamma)=\sum_{\mu\in\Gamma}\xi(x, \mu)f(\mu^{-1} x,\mu^{-1}\gamma), \; \text{ for }\xi\in C_c(X\times \Gamma,H), f\in C_c(X\times\Gamma).
$$
The  $C_c(X\times\Gamma)$-valued inner product is given for $\xi, \eta\in C_c(X\times \Gamma,H)$ by
$$
\langle \xi,\eta\rangle (x,\gamma)=\sum_{\mu\in\Gamma}\langle \xi(\mu x,\mu),\eta(\mu x, \mu\gamma)\rangle.
$$
The completion of $C_c(X\times \Gamma, H)$  is then a Hilbert $C^*(\maG)$-module, that we shall denote by $(C(X)\otimes H)\rtimes\Gamma$. The completion of $C_c(X\times \Gamma, H)$ with respect to the maximal (resp. regular) norm will be denoted $(C(X)\otimes H)\rtimes_{\max}\Gamma$ (resp. $(C(X)\otimes H)\rtimes_r\Gamma$). In the same way, we have the Hilbert module $(C(X)\otimes H)\rtimes_{\ep}\Gamma$ over $C(X)\rtimes_{\ep} \Gamma$. 

\begin{remark}\label{rtimes=otimes}\ 
It is easy to check that this Hilbert $C^*(\maG)$-module is isomorphic to $C^*(\maG)\otimes H$, and hence also to the composition Hilbert module $(C(X)\otimes H)\otimes_{C(X)} C^*(\maG)$, where $C(X)$ acts by left multiplication as a subalgebra of $C^*(\maG)$.
\end{remark}

\medskip

The following lemma will be needed later on. Recall that for $\gamma\in \Gamma$, the characteristic function of $\{\gamma\}$ is denoted $\delta_\gamma:\Gamma\to \{0,1\}$.

\begin{proposition}\label{Descente}\ Denote for $\alpha\in \Gamma$ by $i_\alpha^H: C(X, H)\rightarrow C_c(X\times\Gamma, H)$ and $p_\alpha^H: C_c(X\times\Gamma, H)\rightarrow C(X, H)$ the linear maps defined by
$
i_\alpha^H (\xi) := \xi\otimes\delta_\alpha$ and $p_\alpha^H (\what \xi): x\longmapsto \what\xi (x, \alpha)$. 
Then 
\begin{enumerate}
\item The maps $i_\alpha^H$ and $p_\alpha^H$ extend to bounded operators between the completions 
$$
i_\alpha^H: C(X)\otimes H \rightarrow (C(X)\otimes H)\rtimes\Gamma\quad \text{ and } \quad p_\alpha^H:(C(X)\otimes H)\rtimes\Gamma\rightarrow C(X)\otimes H,
$$
such that $i_\alpha^H$ is an isometry and $p_\alpha^H$ is a contraction.
\item For any adjointable operator ${\what T}\in \maL_{C^*(\maG)} ((C(X)\otimes H)\rtimes\Gamma)$, the operator $T:=p_{1_\Gamma}^H\circ {\what T}\circ i_{1_\Gamma}^H$ is an adjointable operator, i.e. belongs to $\maL_{C(X)} (C(X)\otimes H)$. Moreover, if ${\what T}$ is compact then so is $T$.
\end{enumerate}
\end{proposition}

\begin{proof}\ 

\begin{enumerate}
\item The maps $i_\alpha^H$ and $p_\alpha^H$ are obviously linear. We shall  denote by $i_\alpha$ and $p_\alpha$ the same maps but for $H=\C$. For $\xi\in C(X, H)$ and $\what\xi\in C_c(X\times \Gamma, H)$, we have
$$
\langle i_\alpha^H (\xi) , i_\alpha^H(\xi)\rangle = i_{1_\Gamma} \left( \langle \xi (\alpha\bullet), \xi (\alpha\bullet) \rangle\right) \text{ and } \langle p_\alpha^H({\what\xi}), p_\alpha^H({\what\xi})\rangle (x)= \langle {\what\xi} (x, \alpha),{\what\xi} (x, \alpha)\rangle.
$$
Since $i_{1_\Gamma}$ is an isometry for any choice of admissible crossed product, we deduce that $i_\alpha^H$ extends to  a linear isometry $i_\alpha^H: C(X)\otimes H \hookrightarrow (C(X)\otimes H)\rtimes \Gamma$. 

As for the map $p_\alpha^H$, it is easy to check that  $
\sup_{x\in X} \Vert {\what\xi} (x, \alpha)\Vert_H^2 \leq \Vert {\what\xi}\Vert^2_{C(X,H)\rtimes \Gamma}$.
Indeeed, from the very definition of the reduced  completion norm, we deduce that for any  $x\in X$ and any $\alpha\in \Gamma$,
$$
\Vert \lambda_x \langle {\what\xi}, {\what\xi}\rangle (\delta_{\alpha})\Vert_{\ell^2(\Gamma)} \leq \Vert {\what\xi}\Vert^2_{(C(X)\otimes H)\rtimes_r\Gamma}.
$$
This shows that for any $(x, \alpha)\in \maG$, we have the estimates
$$
\sum_{\gamma\in \Gamma} \left\vert  \sum_{\mu\in \Gamma} \langle{\what\xi}(\mu\gamma x, \mu), {\what\xi} (\mu\gamma x, \mu\gamma\alpha)\rangle   \right\vert^2\; \leq \; \Vert {\what\xi}\Vert^2_{(C(X)\otimes H)\rtimes_r\Gamma}.
$$
In particular, for $\gamma=\alpha^{-1}$, we get
$$
\left\vert \sum_\mu \langle {\what \xi} (\mu\alpha^{-1} x, \mu), {\what\xi}(\mu\alpha^{-1} x, \mu) \rangle  \right\vert^2\; \leq \; \Vert {\what\xi}\Vert^2_{(C(X)\otimes H)\rtimes_r\Gamma}.
$$
The term inside the $\vert\bullet\vert$ is now a sum of non-negative real numbers, and hence for any $\mu\in \Gamma$ and any $x\in X$:
$$
\Vert{\what\xi}(\mu\alpha^{-1} x, \mu) \Vert_H^2 \; \leq \; \Vert {\what\xi}\Vert^2_{(C(X)\otimes H)\rtimes_r\Gamma}.
$$
In particular, for $\mu=\alpha$ this becomes
$$
\Vert{\what\xi}(x, \alpha) \Vert_H^2 \; \leq \; \Vert {\what\xi}\Vert^2_{(C(X)\otimes H)\rtimes_r\Gamma}.\; \leq \; \Vert {\what\xi}\Vert^2_{(C(X)\otimes H)\rtimes\Gamma}
$$
%
Using the previous estimates, we deduce that $p_\alpha^H$ extends to a contractive linear map $p_\alpha^H: (C(X)\otimes H)\rtimes \Gamma \rightarrow C(X)\otimes H$.

\item We have for ${\what \xi}\in C_c(X\times\Gamma, H)$ and $\eta\in C (X, H)$:
$$
\langle p_{1_\Gamma}^H{\what\xi}, \eta \rangle = p_{1_\Gamma} \langle {\what\xi}, i_{1_\Gamma}^H\eta \rangle.
$$
Therefore, for any $\xi\in C(X, H)$,
\begin{eqnarray*}
\langle (p_{1_\Gamma}^H {\what T} i_{1_\Gamma}^H) \xi , \eta\rangle & = & p_{1_\Gamma} \langle {\what T} i_{1_\Gamma}^H \xi, i_{1_\Gamma}^H \eta \rangle\\
& = & p_{1_\Gamma} \langle i_{1_\Gamma}^H \xi,  {\what T}^* i_{1_\Gamma}^H \eta \rangle\\
& = &  \langle  \xi, (p_{1_\Gamma}^H {\what T}^* i_{1_\Gamma}^H) \eta \rangle
\end{eqnarray*}
If now ${\what T}$ ia a compact operator of the Hilbert $C^*(\maG)$-module $(C(X)\otimes H)\rtimes\Gamma$, then it can be approximated by finite sums of operators of the type $ \theta_{{\what \eta}, {\what \eta}'}$ for ${\what \eta}, {\what \eta}'\in C_c(X\times\Gamma, H)$. Since the operators $p_{1_\Gamma}$ and $i_{1_\Gamma}$ are continuous, it suffices to prove that the operator $p_{1_\Gamma}^H\circ \theta_{{\what \eta}, {\what \eta}'} \circ i_{1_\Gamma}^H$ is a compact operator of the Hilbert $C(X)$-module $C(X)\otimes H$. But a straightforward computation shows that 
$$
p_{1_\Gamma}^H\circ \theta_{{\what \eta}, {\what \eta}'} \circ i_{1_\Gamma}^H = \sum_\gamma \theta_{p_\gamma^H({\what\eta}), p_\gamma^H({\what\eta}')}.
$$
{
Notice that the sum is finite since $\widehat{\eta}$ and $\widehat{\eta}'$ both have compact supports.}
\end{enumerate}
\end{proof}

\medskip

We now use the action of $\Gamma$ on $H$. The $\Gamma$-equivariant representation $\pi$ yields a  representation 
$$
\pi\rtimes\Gamma : C_0(Z\times X)\rtimes\Gamma\longrightarrow \maL_{C^*(\maG)}((C(X)\otimes H)\rtimes\Gamma).
$$ 
given for $g\in C_c(Z\times X\times\Gamma)$ and $\xi\in C_c(X\times \Gamma ,H)$ by the formula
$$
(\pi\rtimes \Gamma)(g)(\xi)(x, \gamma)=\sum_{\mu\in\Gamma}\pi_x(g(\bullet, x, \mu)) U_\mu\xi(\mu^{-1} x, \mu^{-1}\gamma).
$$
Recall that $C_0(Z\times X)\rtimes\Gamma$ is our completion  $C^*$-algebra for the transformation groupoid $(Z\times X)\rtimes \Gamma$, and its laws can also be described on the dense subalgebra $C_c(Z\times X\times \Gamma)\simeq C_c(\Gamma, C_c(Z\times X))$ as follows: 
$$
(f*g)(\gamma)(z,x)=\sum_{\mu\in\Gamma}f(\mu)(z,x)g(\mu^{-1}\gamma)(z \mu,\mu^{-1} x)
\mbox{  and  }
f^*(\gamma)(z,x)=\overline{f(\gamma^{-1})(z \gamma,\gamma^{-1} x)}.
$$

Using the chosen compactly supported continuous cutoff function $c:Z\to [0, 1]$, so satisfying $
\sum_{\gamma\in\Gamma} \gamma c^2=1$,
we define a projection $e\in C_c(\Gamma,C_c(Z\times X))$ by setting
$$
e(\gamma)(z,x)=c(z)c(z \gamma), \text{ for } z\in Z, x\in X\text{ and }\gamma\in \Gamma.
$$
It is indeed easy to check that $e*e=e = e^*$. The projection $e$  is the Mishchenko projection associated with $c$.
Notice that $(\pi\rtimes\Gamma)(e)((C(X)\otimes H)\rtimes\Gamma)$,  the range of the projection $(\pi\rtimes\Gamma)(e)$, is then a  right Hilbert $C^*(\maG)$-module as well and we now state   that it provides another  description of our Connes-Skandalis Hilbert module $\maE^H$.

\begin{lemma}
\label{Smough2}
There exists an isomorphism of Hilbert $C^*(\maG)$-modules
$$
\phi=\phi^H : \maE^H\overset{\simeq}\longrightarrow(\pi\rtimes\Gamma)(e)((C(X)\otimes H)\rtimes\Gamma)
$$
induced by the formula
$$
\phi(\xi)(x, \gamma)=\pi_x(c)U_\gamma\xi(\gamma^{-1} x), \quad  \text{for } \xi\in \maE_c^H\text{ and }(\gamma, x)\in \maG.
$$
\end{lemma}

\begin{proof}\
We have for $\xi\in \maE_c^H$ and $(\gamma, x)\in \maG$:
$$
(\pi\rtimes\Gamma)(e)(\phi(\xi))(x, \gamma) =
 \sum_{\mu\in\Gamma}\pi_x(c)\pi_x(\mu c)U_\mu\pi_{\mu^{-1}x}(c)U_{\mu^{-1}\gamma}\xi((\mu^{-1}\gamma)^{-1}\mu^{-1} x).
 $$
 But, $U_\mu\circ \pi_{\mu^{-1}x} (c)=\pi_x (\mu c) \circ U_\mu$ and $\sum_\mu \mu c^2=1$, hence we deduce
$$
(\pi\rtimes\Gamma)(e)(\phi(\xi))(x, \gamma) = \phi(\xi)(x, \gamma)\text{ and so }\phi(\xi)\in(\pi\rtimes\Gamma)(e)(C_c(\Gamma\times X, H)).
$$
Moreover, we have $\langle\phi(\xi),\phi(\xi)\rangle =\langle\xi,\xi\rangle$ by direct verification. 
If now $\eta\in C_c(\Gamma\times X, H)$ is  given, then by setting
$$
\xi(x)=\sum_{\mu\in\Gamma}\pi_x(\mu c)U_\mu\eta(\mu^{-1} x, \mu^{-1}),
$$
we define a preimage of $\eta$ under $\phi$, since clearly $\xi\in\maE_c^H$ and we can compute
\begin{eqnarray*}
\phi(\xi)(x, \gamma) &=& \pi_x(c)U_\gamma\xi(\gamma^{-1}x)\\
&=&\sum_{\mu\in\Gamma}\pi_x(c)\pi_x(\mu c)U_\mu\eta(\mu^{-1} x, \mu^{-1}\gamma)\\
&=& (\pi\rtimes\Gamma)(e)(\eta)(x, \gamma).
\end{eqnarray*}
\end{proof}

Thanks to  Lemma \ref{Smough2},  we can describe the Connes-Skandalis Hilbert module   as the  composition Hilbert module (see for instance \cite{Lance})
$$
e(C_0(Z\times X)\rtimes\Gamma)\otimes_{\pi\rtimes\Gamma}((C(X)\otimes H)\rtimes\Gamma) \text{ also denoted } e(C_0(Z\times X)\rtimes\Gamma)\otimes_{C_0(Z\times X)\rtimes\Gamma}((C(X)\otimes H)\rtimes\Gamma).
$$
More precisely, we have the following standard

\begin{lemma}\label{CS general}
The following formula
$$
\alpha(f\otimes\xi)=(\pi\rtimes\Gamma)(f)\, (\xi)\quad \text{for  }f\in e\left(C_0(Z\times X)\rtimes \Gamma\right)\text{ and }\xi\in (C(X)\otimes H)\rtimes \Gamma,
$$
induces  an isomorphism of Hilbert $C^*(\maG)$-modules
$$
\alpha : e\left[(C_0(Z\times X)\rtimes\Gamma\right]\otimes_{\pi\rtimes\Gamma}\left[(C(X)\otimes H)\rtimes\Gamma\right]\overset{\simeq}\longrightarrow(\pi\rtimes\Gamma)(e)\left[(C(X)\otimes H)\rtimes\Gamma\right].
$$
Therefore, we get an isomorphism of Hilbert modules
$$
\alpha^{-1}\circ \phi: \maE^H\stackrel{\simeq} {\longrightarrow}e\left[(C_0(Z\times X)\rtimes\Gamma\right]\otimes_{\pi\rtimes\Gamma}\left[(C(X)\otimes H)\rtimes\Gamma\right].
$$
\end{lemma}

\begin{proof}
The morphism  $\alpha$ is well defined since $\pi\rtimes\Gamma$ is a homomorphism.
Moreover, we have
\begin{eqnarray*}
\langle\alpha(f\otimes\xi),\alpha(f\otimes\xi)\rangle&=&\langle(\pi\rtimes\Gamma)(f) (\xi),(\pi\rtimes\Gamma)(f)(\xi)\rangle\\
&=&\langle f\otimes\xi,f\otimes\xi\rangle.
\end{eqnarray*}
If now  $\eta\in (C(X)\otimes H)\rtimes \Gamma$, then  $
\alpha(e\otimes\eta)=(\pi\rtimes\Gamma)(e)(\eta),$
which shows the surjectivity of  $\alpha$.
\end{proof}

\begin{remark}
\label{Ornstein2}
A straightforward computation gives the following formula for the isomorphism $
\phi^{-1}\circ\alpha$:
$$
(\phi^{-1}\circ\alpha)(f\otimes\xi)(x)=
\sum_{\gamma,\mu\in\Gamma}\pi_x(\mu c)\pi_x(\mu g(\mu^{-1}\gamma)_{\mu^{-1} x})U_\gamma\xi(\gamma^{-1},\gamma^{-1} x).
$$
\end{remark}

\medskip

\section{The admissible HR sequence }

Most of the constructions we give below are valid in general, we shall though sometimes assume that the compact space  $X$ is  finite dimensional, especially when we apply   the Paschke-Higson duality theorem as proved in \cite{BenameurRoyII}.  This will then be indicated. 

\subsection{Roe algebras}

Recall that if $T\in\maL_{C(X)}(C(X)\otimes H)$, then the operator $T\rtimes\Gamma\in\maL_{C^*(\maG)}((C(X)\otimes H)\rtimes\Gamma)$ is defined by
$$
(T\rtimes\Gamma)(\xi)(x,\gamma)=T_x\xi(x,\gamma).
$$
Said differently, if we use the isomorphism $(C(X)\otimes H)\rtimes \Gamma\simeq (C(X)\otimes H)\otimes_{C(X)} C^*(\maG)$ then the operator $T\rtimes \Gamma$ is just the composition operator $T\otimes_{C(X)} \Id$. In particular, the map 
$$
\maL_{C(X)}(C(X)\otimes H) \longrightarrow \maL_{C^*(\maG)}((C(X)\otimes H)\rtimes \Gamma)\text{ given by }T\longmapsto T\rtimes \Gamma,
$$
is a $C^*$-algebra homomorphism. Moreover, 
and since the morphism $C(X)\to C^*(\maG)$ is injective, this morphism is also injective and hence isometric. Finally, an easy verification shows that if $T$ is a compact operator of the Hilbert $C(X)$-module $C(X)\otimes H$ then $T\rtimes\Gamma\simeq {{T\otimes_{C(X)}\id}}$ is a compact operator of the Hilbert $C^*(\maG)$-module $(C(X)\otimes H)\rtimes \Gamma$. For more details on these properties, see for instance \cite{Lance}[Proposition 4.7], and \cite{BaumGuentnerWillett} for the extension to admissible crossed products. 

\medskip

\begin{corollary}\label{rtimesCompact}
If $T\in \maL_{C(X)} (C(X)\otimes H)$ is such that  the operator ${\what T}:=T\rtimes \Gamma\simeq {{T\otimes_{C(X)}\id}}$ is a compact operator of the $C^*(\maG)$-Hilbert module $(C(X)\otimes H)\rtimes\Gamma$, then the operator $T$ is a compact operator of the $C(X)$-Hilbert module $C(X)\otimes H$.
\end{corollary}

\begin{proof}\
We apply Proposition \ref{Descente} and notice that $p_{1_\Gamma}^H\circ (T\rtimes\Gamma) \circ i_{1_\Gamma}^H = T$.
\end{proof}

\medskip

We shall use the following two equivalent definitions of lifts of operators. Notice  in particular that our definitions rely a priori on the choice of a cut-off function $c$ and use the associated projection $e\in C_0(X\times Z)\rtimes \Gamma$. However, we explain below that the  resulting admissible  Roe $C^*$-algebras do not depend on such choice.  

Recall that for for any chosen cut-off function $c$, and for $f\in e(C_0(Z\times X)\rtimes\Gamma)$, the operator  $L_f$  is 
$$
L_f: (C(X)\otimes H)\rtimes \Gamma \rightarrow e(C_0(Z\times X)\rtimes\Gamma)\otimes_{\pi\rtimes\Gamma} \left((C(X)\otimes H)\rtimes \Gamma\right)\text{ given by }L_f (u)=f\otimes u.
$$
This should cause no confusion and $L_f$ can also be composed with the identification \ref{Ornstein2} so that it is valued in the  Connes-Skandalis Hilbert module $\maE^H$. We shall still denote this composite map by $L_f$ for simplicity.

\medskip

\begin{definition}[Lifts of operators, first definition]
\label{def1 lift2}
Let $T\in\maL_{C(X)}(C(X)\otimes H)$ and assume that $T$ commutes modulo compact operators with the representation $\pi$  of $C_0(Z\times X)$.
An operator $T_\maE\in \maL_{C^*(\maG)}(\maE^H)$ is called a \emph{lift} of $T$ if for some cut-off function $c\in C_c(Z)$ with its associated Mishchenko projection $e$, $T_\maE$ is a lift of $T\rtimes\Gamma\in\maL_{C^*(\maG)}((C(X)\otimes H)\rtimes\Gamma)$ in the sense of definition \ref{lift}. So via the identification \ref{Ornstein2} using $e$, this means that  for any $f\in e(C_0(Z\times X)\rtimes\Gamma)$, the following diagrams commute modulo compact operators:
$$\xymatrix{
{\maE^H} \ar@{->}[r]^{T_\maE} \ar@{<-}[d]_{L_f}  & {\maE^H} \ar@{<-}[d]^{L_f} 
&\mbox{ and } &
{\maE^H} \ar@{->}[r]^{T_\maE} \ar@{->}[d]_{(L_f)^*}  & {\maE^H} \ar@{->}[d]^{(L_f)^*}\\
{(C(X)\otimes H)\rtimes\Gamma} \ar@{->}[r]^{T\rtimes\Gamma}&
{(C(X)\otimes H)\rtimes\Gamma}
& & {(C(X)\otimes H)\rtimes\Gamma} \ar@{->}[r]^{T\rtimes\Gamma}&
{(C(X)\otimes H)\rtimes\Gamma} 
}$$
\end{definition}

\medskip

We point out that if $T_\maE$ is a lift of $T$ in the sense of Definition \ref{def1 lift2} for some cut-off function, then it is automatically a lift of $T$ for any cut-off function. Therefore, Definition \ref{def1 lift2}  is independent of the choice of the cut-off function $c$, see Remark \ref{c-independent} below. Also, notice that if $T_\maE$ is a given lift for $T$ then  any compact perturbation of $T_\maE$ gives another lift of $T$. 

By using the identification of Lemma \ref{CS general}, we can also use    the following equivalent definition of a lift.

\medskip

\begin{definition}[Lifts of operators, second definition]
\label{def2 lift2}
We say that an operator $T_\maE\in \maL_{C^*(\maG)}(\maE^H)$ is a \emph{lift} of an operator $T\in\maL_{C(X)}(C(X)\otimes H)$ if for any $f\in e(C_0(Z\times X)\rtimes\Gamma)$, the following diagrams commute modulo compact operators  (using the identification \ref{Smough2}):
$$\xymatrix{
{\maE^H} \ar@{->}[r]^{T_\maE} \ar@{<-}[d]_{(\pi\rtimes\Gamma)(f)}  & {\maE^H} \ar@{<-}[d]^{(\pi\rtimes\Gamma)(f)} 
&\mbox{ and } &
{\maE^H} \ar@{->}[r]^{T_\maE} \ar@{->}[d]_{(\pi\rtimes\Gamma)(f^*)}  & {\maE^H} \ar@{->}[d]^{(\pi\rtimes\Gamma)(f^*)}\\
{(C(X)\otimes H)\rtimes\Gamma} \ar@{->}[r]^{T\rtimes\Gamma}&
{(C(X)\otimes H)\rtimes\Gamma}
& & {(C(X)\otimes H)\rtimes\Gamma} \ar@{->}[r]^{T\rtimes\Gamma}&
{(C(X)\otimes H)\rtimes\Gamma} 
}$$
\end{definition}

\medskip

Again and by the similar argument explained in Remark \ref{c-independent}, a lift of $T$ in the sense of Definition \ref{def2 lift2} relative to some cut-off function will be a lift of $T$ relative to any cut-off function. Hence, this second definition also does not depend on the choice of cut-off function $c\in C_c(Z)$. 
Definitions \ref{def1 lift2} and \ref{def2 lift2} of a lift are actually obviously equivalent.
Notice indeed that we have for any $\xi\in C_c(X\times\Gamma,H)$ and any $g\in C_c(\Gamma,C_0(Z\times X))$, 
$(\alpha\circ L_f)(\xi)=(\pi\rtimes\Gamma)(f)\xi$. Moreover, 
$$
L_{f^*}(g\otimes\xi)=(\pi\rtimes\Gamma)(f^*g)(\xi)=(\pi\rtimes\Gamma)(f^*)(\pi\rtimes\Gamma)(g)\xi=(\pi\rtimes\Gamma)(f^*)\alpha(g\otimes\xi).
$$

\medskip

We are now in position to introduce the Roe algebra $D^*_\Gamma (\maE^H)$ associated with the fixed admissible crossed product functor, as well as its ideal $C^*_\Gamma (\maE^H)$. Recall the equivariant Roe algebra $D^*_\Gamma (X, Z, H)$ and its ideal $C^*_\Gamma (X, Z, H)$ {{of Definition \ref{EquivariantRoeAlgebras}}}. 

\medskip

\begin{definition}\label{DefRoe} The  Roe algebras associated with $Z$, with the $\Gamma$-equivariant representation in $H$ and with our admissible crossed product, are defined as follows:
\begin{itemize}
\item The Roe algebra $D^*_\Gamma (\maE^H)$ is the $C^*$-algebra closure  in $\maL_{C^*(\maG)} (\maE^H)$ of all operators that are lifts of operators in $D^*_\Gamma (X, Z, H)$.
\item The Roe ideal $C^*_\Gamma (\maE^H)$ is the $C^*$-algebra composed of operators  in $\maL_{C^*(\maG)} (\maE^H)$ which are lifts of operators in $C^*_\Gamma (X, Z, H)$.
\end{itemize}
\end{definition}

\medskip

Composition of lifts is clearly a lift for the composition of the original operators, and similarly for the adjoints, hence $D^*_\Gamma (\maE^H)$ is a unital $C^*$-algebra and  $C^*_\Gamma (\maE^H)$ is a two sided involutive closed ideal of $D^*_\Gamma(\maE^H)$. 
We shall prove below (see Proposition \ref{locally compact2}) that  $C^*_\Gamma (\maE^H)$ is nothing but  the ideal $\maK_{C^*(\maG)}(\maE^H)$ of compact operators.

\begin{remark}
When $X=\{\bullet\}$ is reduced to a singleton, although not totally obvious, Definition \ref{DefRoe} of the Roe $C^*$-algebras is compatible up to isomorphism with the previous definitions of Higson and Roe \cite{HigsonRoe2008}. More precisely, when the admissible crossed product functor $\rtimes\Gamma=\rtimes_r\Gamma$ is the reduced one, we get the reduced Roe algebra, while we obtain the maximal one introduced and studied in \cite{HigsonRoe2008} when $\rtimes\Gamma = \rtimes_{\max}\Gamma$. These verifications  are  proved by the second author in \cite{VictorThesis}.
\end{remark}

\color{black}

In the sequel, we shall denote as usual for $\xi\in C(X, H)$, ${\what \xi}\in C_c(X\times\Gamma, H)$ and $\gamma\in\Gamma$, by $\gamma \xi$ and $\gamma{\what\xi}$ the action of $\gamma$ on $\xi$ and ${\what\xi}$ respectively, say the elements of $C(X, H)$ and $C_c(X\times \Gamma, H)$ defined by 
$$
(\gamma\xi)(x)=U_\gamma (\xi (\gamma^{-1}x))\text{ and } (\gamma{\what\xi}) (x, \alpha)=U_\gamma (\xi (\gamma^{-1}x, \gamma^{-1}\alpha)).
$$

For a $\Gamma$-invariant operator $T\in\maL_{C(X)}(C(X)\otimes H)^\Gamma$, recall the following  facts:
\begin{lemma}\label{rtimes}\
\begin{enumerate}
\item If $T$ is $\pi$-pseudolocal then the operator $T\rtimes\Gamma$ is $\pi\rtimes \Gamma$-pseudolocal. 
\item $T$ is $\pi$-locally compact if and only if the operator $T\rtimes\Gamma$ is $\pi\rtimes\Gamma$-locally compact. 
\end{enumerate} 
%
\end{lemma}



\begin{proof}\ 
We have for  $f\in C_c(Z\times X\times\Gamma)$,  $T\in\maL_{C(X)}(C(X)\otimes H)^\Gamma$, $\xi\in C_c(X\times \Gamma, H)$ and $(x, \gamma)\in \maG$:
\begin{eqnarray*}
\left((\pi\rtimes\Gamma)(f)\circ (T\rtimes\Gamma)\right)(\xi)(x, \gamma)&=&\sum_{\mu\in\Gamma}\pi_x(f(\bullet, x, \mu))U_\mu T_{\mu^{-1} x}\left(\xi(\mu^{-1}x, \mu^{-1}\gamma)\right)\\
&=&\sum_{\mu\in\Gamma}\pi_x(f(\bullet, x, \mu))T_xU_\mu \left(\xi(\mu^{-1} x, \mu^{-1}\gamma)\right)
\mbox{  (because $T$ is $\Gamma$-equivariant)}
\end{eqnarray*}
while 
\begin{eqnarray*}
\left((T\rtimes\Gamma)\circ (\pi\rtimes\Gamma)(f)\right) (\xi)(x, \gamma)&=&T_x\left((\pi\rtimes\Gamma)(f)(\xi)(x,\gamma)\right)\\
&=&\sum_{\mu\in\Gamma}T_x\pi_x(f(\bullet, x, \mu))U_\mu\left(\xi(\mu^{-1} x,\mu^{-1}\gamma)\right)
\end{eqnarray*}
Hence we have 
\begin{eqnarray*}
\left[(\pi\rtimes\Gamma)(f), T\rtimes\Gamma\right] (\xi)(x, \gamma)&=&
\sum_{\mu\in\Gamma}\left[\pi_x(f(\bullet, x, \mu)) , T_x\right] U_\mu\left(\xi(\mu^{-1}x,\mu^{-1}\gamma)\right)
\end{eqnarray*}
If $T$ is locally compact, then each $\pi(f(\bullet, \bullet, \mu)) T$ and $T\pi(f(\bullet, \bullet, \mu))$ are compact operators on $C(X)\otimes H$. Similarly if $T$ is pseudolocal each $[\pi(f(\bullet, \bullet, \mu)), T]$ is a compact operator on $C(X)\otimes H$. 
This  easily allows to conclude.  
Let us prove now the converse in the second  item. 
Assume  that $T$ is now a $\Gamma$-invariant operator such that the operator $T\rtimes \Gamma$ is locally compact. This means that for any ${\what f}\in C_0(Z\times X)\rtimes\Gamma$, 
$$
(\pi\rtimes\Gamma)({\what f})(T\rtimes\Gamma)\text{ and }(T\rtimes\Gamma)(\pi\rtimes\Gamma)({\what f}),
$$ 
are compact operators of the Hilbert module $(C(X)\otimes H)\rtimes \Gamma$.

Let $f$ be an element of $C_c(Z\times X)$ and set ${\what f}:=f\otimes\delta_{1_\Gamma}$. 
We may apply the assumption to the function ${\what f}$ using the relation
$$
(\pi\rtimes \Gamma) ({\what f})\circ (T\rtimes \Gamma) = (\pi (f) \circ T)\rtimes \Gamma\text{ since }(\pi\rtimes \Gamma) ({\what f}) =\pi (f)\rtimes \Gamma.
$$
More precisely, we get  for any ${\what \xi}\in C_c(X\times \Gamma, H)$,  $(\pi\rtimes \Gamma) ({\what f}) (T\rtimes \Gamma) ({\what \xi} (x, \gamma)) = \pi_x(f(\bullet, x))T_x (\xi(x, \gamma))$. We are thus reduced to proving that if an element $S$ of $\maL_{C(X)} (C(X)\otimes H)$ satisfies that $S\rtimes \Gamma$ is $C^*(\maG)$-compact, then it is itself $C(X)$-compact. But this is the content of the second item of Proposition \ref{Descente} so that the proof is now complete.

\end{proof}

The following  stronger criterion is used in the proof of Proposition  \ref{equivalence2}. Recall that $e$ denotes the Mishchenko projection associated with our fixed cut-off function $c$.

\begin{lemma}\label{rtimesMishchenko}
A $\Gamma$-invariant operator $T\in\maL_{C(X)}(C(X)\otimes H)^\Gamma$  is locally compact if and only if the operators $(\pi\rtimes\Gamma)(e)(T\rtimes \Gamma)$ and $(T\rtimes \Gamma)(\pi\rtimes\Gamma)(e)$ are compact. \end{lemma}

\begin{proof}\ 
%
%
By using the second item of Lemma \ref{rtimes},  we only need to prove that if $T$ satisfies that
$$
(\pi\rtimes\Gamma)(e)(T\rtimes \Gamma)\text{ and }(T\rtimes \Gamma)(\pi\rtimes\Gamma)(e)\text{  are compact operators},
$$
then $T$ is locally compact. 
Recall that 
$$
e(z,  \gamma):=c(z) \, c(z\gamma).
$$
Let $f\in C_c(Z\times X)$ be a given continuous compactly supported function and let $\mu\in \Gamma$ be fixed. Then the function $f\otimes\delta_\mu$ belongs to $C_c(Z\times X\times \Gamma)$ and we may apply our assumption to deduce that $(\pi\rtimes\Gamma) ((f\otimes\delta_\mu)* e) (T\rtimes \Gamma)$ and $(T\rtimes \Gamma) (\pi\rtimes\Gamma) (e*(f\otimes\delta_\mu)) $ are compact operators. But we can compute, for the first relation for instance, and for $\xi\in C(X)\otimes H$ and $\gamma\in \Gamma$, then we get using that $T$ is $\Gamma$-invariant:
$$
(\pi\rtimes\Gamma) ((f\otimes\delta_\mu)* e) (T\rtimes \Gamma) (\xi\otimes\delta_{1_\Gamma}) (x, \gamma) = \pi_x (f (\bullet, x) (\mu c)(\gamma c)) T_x U_\gamma (\xi (\gamma^{-1}x)).
$$
If we compose on the left with the operator $\pi(\mu c)\rtimes \Gamma$, and we sum up over $\mu$, then only a finite number of such elements of $\Gamma$ are involved and we get 
$$
\sum_{\mu\in \Gamma} \left[ (\pi_x (\mu c)^2 \pi_x(\gamma c)\pi_x(f(\bullet, x))T_x U_\gamma (\xi (\gamma^{-1}x)) \right] = \pi_x (f(\bullet, x)(\gamma c)) T_x (U_\gamma\xi(\gamma^{-1}x)).
$$
Therefore, if we denote by ${\what T}^f$ the compact operator  
$$
{\what T}^f  := \sum_\mu \left[(\pi (\mu c)\rtimes\Gamma) (\pi\rtimes\Gamma)((f\otimes\delta_\mu)* e) (T\rtimes \Gamma)  \right],
$$
then the following equality holds 
$$
{\what T}^f (\xi\otimes \delta_{1_\Gamma}) = \sum_{\gamma} A_\gamma^{f, T} (\xi) \otimes\delta_\gamma,
$$
where $A_\gamma^{f, T}= \pi(f (\gamma c)) \circ T \circ \gamma\in \maL_{C(X)} (C(X)\otimes H)$, with the notation $\gamma (\xi) (x)=U_\gamma (\xi(\gamma^{-1} x))$. Notice that this latter operator $\gamma$ is not adjointable, and 
recall  as well the isometric embeddings $i_\gamma^H:C(X)\otimes H\hookrightarrow C^*(\maG)\otimes H$ given by $\xi\mapsto \xi\otimes \delta_\gamma$, and the contractions  $p_\gamma^H:C^*(\maG)\otimes H\to C(X)\otimes H$ given by $\what\xi\mapsto \xi (\bullet, \gamma)$, both introduced in Proposition \ref{Descente} and not adjointable.  We see from the above verifications that  the operator 
$$
p_\gamma^H\circ {\what T}^f\circ i_{1_\Gamma}^H\circ \gamma^{-1},
$$
is adjointable, it actually coincides with $\pi(f (\gamma c)) \circ T$. Moreover, the condition that ${\what T}^f$ is compact implies that $\pi(f (\gamma c)) \circ T$ is also compact.  Indeed, if the adjointable operator ${\what T}^f$  is approximated by finite sums of operators $\theta_{{\what \eta_i}, {\what \eta_i'}}$ then the operator $p_\gamma^H\circ {\what T}^f\circ i_{1_\Gamma}^H\circ \gamma^{-1}$ will be approximated by finite sums of operators of the type
$$
p_\gamma^H\circ\theta_{{\what \eta_i}, {\what \eta_i'}} \circ i_{1_\Gamma}^H\circ \gamma^{-1}.
$$
Now, a direct inspection shows that each
$$
p_\gamma^H\circ\theta_{{\what \eta}, {\what \eta'}} \circ i_{1_\Gamma}^H\circ \gamma^{-1} = \sum_\alpha \theta_{p_\alpha^H{\what\eta}, p_\alpha^H (\gamma{\what\eta}')},
$$
where the sum is finite, and hence $p_\gamma^H\circ {\what T}^f\circ i_{1_\Gamma}^H\circ \gamma^{-1}$ is a compact adjointable operator of the Hilbert $C(X)$-module $C(X)\otimes H$. This proves that for any $\gamma\in \Gamma$, the operator $\pi(f (\gamma c)) \circ T$ is  compact. Composing again  with $\pi (\gamma c)$ on the left and summing up over the finite set of contributing elements $\gamma\in \Gamma$ (recall that $f$ is compactly supported and $\Gamma$ acts properly on $Z$), we get
$$
\pi(f) \circ T \text{ is  a compact operator as allowed}.
$$
A similar argument, now starting with the hypothesis that $(T\rtimes \Gamma) (\pi\rtimes\Gamma) (e * (f\otimes\delta_\mu))$ is compact,  allows to prove that $T\circ \pi (f)$ is also a compact operator. 
\end{proof}

\medskip

\begin{proposition}\label{locally compact2}
With  the previous notations, we have
\begin{enumerate}
\item[•] For any $T\in D^*_\Gamma (X, Z, H)$, a lift of $T$ is a compact perturbation of the  operator $T_e\in \maL_{C^*(\maG)} (\maE^H)$ which is the conjugate  by the isomorphism $\phi$ of Lemma \ref{Smough2} of $(\pi\rtimes\Gamma)(e)(T\rtimes\Gamma)(\pi\rtimes\Gamma)(e)$.
\item[•] Moreover,  $C^*_\Gamma (\maE^H)$ coincides with the ideal of $C^*(\maG)$-compact operators, i.e. $C^*_\Gamma (\maE^H) = \maK_{C^*(\maG)}(\maE^H)$.
\end{enumerate}
\end{proposition}

\medskip

\begin{proof}
Let $f\in e\left(C_0(Z\times X)\rtimes\Gamma\right)$, then $ef=f$ and we may write
$$
(\pi\rtimes\Gamma)(e)\left[T\rtimes\Gamma, (\pi\rtimes \Gamma)(f)\right] = (\pi\rtimes\Gamma)(e)(T\rtimes \Gamma)(\pi\rtimes\Gamma)(f) - (\pi\rtimes\Gamma)(f)(T\rtimes\Gamma).
$$
The commutator in the LHS is  a compact operator by Lemma \ref{rtimes}. Therefore the RHS is also a compact operator. {{Recall that the notation $(\pi\rtimes\Gamma)(f)$ in Definition \ref{def2 lift2} really means $\phi^{-1}  (\pi\rtimes\Gamma)(f)$ where $\phi$ is the isomorphism of Lemma \ref{Smough2}. Hence we deduce by definition of $T_e$ that}}
{{\begin{multline*}
\phi \left[T_e  (\phi^{-1}  (\pi\rtimes\Gamma)(f))  - (\phi^{-1}(\pi\rtimes\Gamma)(f)) (T\rtimes \Gamma)\right] \\ =  (\pi\rtimes\Gamma)(e)(T\rtimes \Gamma)(\pi\rtimes\Gamma)(e) (\pi\rtimes\Gamma)(f)  - (\pi\rtimes\Gamma)(f) (T\rtimes\Gamma)\\
=(\pi\rtimes\Gamma)(e)(T\rtimes \Gamma) (\pi\rtimes\Gamma)(f)  - (\pi\rtimes\Gamma)(f) (T\rtimes\Gamma)
\end{multline*}
So $T_e  (\phi^{-1}  (\pi\rtimes\Gamma)(f))  - (\phi^{-1}(\pi\rtimes\Gamma)(f)) (T\rtimes \Gamma)$ is a compact operator, and we have shown that the first diagram in Definition \ref{def2 lift2} commutes up to compact operators. Replacing $T$ by $T^*$ in this computation shows that   $T_e$ is indeed a lift of $T$.}} Now, notice that the operator $L_e\circ (L_e)^*$ is the identity of the Hilbert module range of $e$.  Hence, if $\what T$ is a lift of the zero operator, then 
$$
{\what T} = {\what T}\circ L_e\circ (L_e)^*  =( {\what T}\circ L_e - L_e\circ 0)\circ (L_e)^* \text{ is compact}.
$$
The second item is also proved similarly by using Lemma \ref{rtimes} and conjugation by the isomorphism $\phi$. More precisely, if $T\in C^*_\Gamma(X,Z,H)$, then $T\rtimes \Gamma$ is $\pi\rtimes\Gamma$-locally compact by Lemma \ref{rtimes}. Therefore, the operator $(\pi\rtimes\Gamma)(e)(T\rtimes\Gamma)$ is a compact operator on the Hilbert module $(C(X)\otimes H)\rtimes\Gamma$. On the other hand, any compact operator is a lift of any $T\in C^*_\Gamma (X, Z, H)$ since $L_f\circ T$ is then a compact operator. This ends the proof.
\end{proof}

Hence, the \emph{Roe-algebra} $D^*_\Gamma(\maE^H)$ of Definition \ref{DefRoe} associated with $Z$ and the $\Gamma$-equivariant representation in $H$ could as well be defined as the $C^*$-algebra closure in $\maL_{C^*(\maG)}(\maE^H)$ of the involutive algebra 
$$
\{T_e+C, T\in D^*_\Gamma (X, Z, H) \text{ and } C\in \maK_{C^*(\maG)} (\maE^H)\}.
$$
In the same way,  $C^*_\Gamma (\maE^H)$ could as well be defined as the ideal of compact operators on the  Connes-Skandalis Hilbert module $\maE^H$, but our description in terms of lifts will be useful in the sequel.

\medskip

\begin{remark}\
 It is understood in our missleading but simplified notations $``\maE^H"$ and  $``D^*_\Gamma (\maE^H)"$ that the cocompact $\Gamma$-proper  space $Z$ with the $\Gamma$-equivariant $C(X)$-representation $\pi$ of $C_0(Z\times X)$ are chosen. Moreover, the admissible crossed product functor is also implicit, so that our definitions include for instance  the reduced, the maximal, as well as the BGW crossed products.
 \end{remark}
 
 \begin{remark}\
Notice that if we use the reduced crossed product functor then we obtain the Roe algebra $D^*_\Gamma (Z, X, H)$  with its ideal $C^*_\Gamma (Z, X, H)$. 
 \end{remark}

 \begin{remark}\label{c-independent}\
 Suppose that $c'\in C_c(Z)$ is another cutoff function and denote by $e'$ the Mishchenko projection defined using $c'$. One  then defines  for  $\xi\in(\pi\rtimes\Gamma)(e)(C_c(\Gamma\times X,H))$ an element $\pphi(\xi)\in (\pi\rtimes\Gamma)(e')(C_c(\Gamma\times X,H))$ by setting
$$
\pphi(\xi)(\gamma,x)=\sum_{\mu\in\Gamma}\pi_x(c')\pi_x(\mu c)U_\mu\xi(\mu^{-1}\gamma,\mu^{-1}x),
$$
and it is then easy to check that $\pphi$ extends to a Hilbert module isomorphism  $
\pphi:(\pi\rtimes\Gamma){{(e)}}((C(X)\otimes H)\rtimes\Gamma)\overset{\simeq}\longrightarrow(\pi\rtimes\Gamma)(e')((C(X)\otimes H)\rtimes\Gamma)$
which intertwines the corresponding Roe algebras, for any admissible crossed product functor. 
\end{remark}

\medskip

We  have  the following important corollary of Proposition \ref{locally compact2}. 
%

\begin{proposition}
\label{equivalence2}\
Denoting by $[T_e]$ the class in the quotient Roe algebra $Q^*_\Gamma(\maE^H) = D^*_\Gamma(\maE^H)\,/\, \maK (\maE^H)$ of the operator $T_e$, the  map $T\mapsto [T_e]$ induces a $C^*$-algebras isomorphism
$$
\theta: Q^*_\Gamma (Z, X, H) \longrightarrow Q^*_\Gamma(\maE^H).
$$
Said differently, the quotient $C^*$-algebra $Q^*_\Gamma(\maE^H)$ does not depend, up to canonical isomorphism,  on the choice of admissible crossed product.
\end{proposition}

\begin{proof}\
The map 
$D^*_\Gamma(Z, X, H) \rightarrow  Q^*_\Gamma(\maE^H)$ defined  by $T \mapsto [T_e]$ is easily seen to be a $C^*$-algebra homomorphism, using that the operator $T\rtimes\Gamma$ is $\pi\rtimes\Gamma$-pseudolocal, thanks to the first item of Lemma \ref{rtimes}. Using the first item of Proposition \ref{locally compact2}, we observ that $\theta$ is onto. Using Proposition \ref{locally compact2} again, we deduce that $\Ker(\theta)$ contains $C^*_\Gamma(X,Z,H)$. If we assume that $T_e$ is a compact operator of the Hilbert module $\maE^H$, then the operators $(\pi\rtimes\Gamma)(e) (T\rtimes \Gamma)$ and $(T\rtimes \Gamma)(\pi\rtimes\Gamma)(e)$ are compact. {{Indeed, since $T_e$ is a lift of $T$, in particular the operator 
$$
T_e (\phi^{-1} (\pi\rtimes\Gamma)(e)) - (\phi^{-1}(\pi\rtimes\Gamma)(e))(T\rtimes\Gamma)
$$ 
is a compact operator. This shows that $(\pi\rtimes\Gamma)(e)(T\rtimes \Gamma)$ is compact, and hence also that $(\pi\rtimes\Gamma)(e)(T\rtimes \Gamma)$ is compact.}} Since $T$ is $\Gamma$-invariant, we may apply  Lemma \ref{rtimesMishchenko} to conclude that $T$ is $\pi$-locally compact, and eventually that $T$ belongs to $C^*_\Gamma (Z, X, H)$. 
Therefore, the induced map $\theta:Q^*_\Gamma(X,Z,H) \to Q^*_\Gamma(\maE^H)$ is a $C^*$-algebra isomorphism  as claimed.
\end{proof}

\medskip

\begin{definition}[The HR-sequence for $Z, H$]\label{MaximalHRsequence}\
Associated with our previous $\Gamma$-equivariant data $(Z, H, \pi)$ over the groupoid $\maG=X\rtimes \Gamma$ and with the admissible crossed product functor $\rtimes\Gamma$, the  Higson-Roe short exact sequence of $C^*$-algebras
is the following exact sequence obtained via the above identifications
$$
0\to \maK_{C^*(\maG)}(\maE^H)\hookrightarrow D^*_\Gamma(\maE^H)\longrightarrow  {{Q^*_\Gamma (Z, X, H)}}\to 0.
$$
\end{definition}

\medskip

\begin{remark}\
Again, although not specified in our notation, it is worth pointing out  that $\maE^H$ as well as $D^*_\Gamma(\maE^H)$ and $Q^*_\Gamma(\maE^H)$ do depend on $Z$ and on the $\Gamma$-equivariant $C(X)$-representation $\pi$. 
\end{remark}

\medskip

\color{black}

\subsection{$K$-theory of admissible Roe algebras}

We shall now prove that, if we amplify  the representation $\pi$, say if we replace it by the very ample representation in the Hilbert space $H\otimes\ell^2(\Gamma)^\infty$ (where $\ell^2(\Gamma)^\infty=\ell^2(\Gamma)\otimes \ell^2(\N)$),  the K-theory groups $K_*(D^*_\Gamma(\maE^H))$ become  independent of the choice of the representation.

%
%

Let us fix unitary representations $U :\Gamma\to\maU(H)$ and $U' :\Gamma\to\maU(H')$, and  {(fiberwise) faithful and} non-degenerate $\Gamma$-equivariant $C(X)$-representations  $\pi: C_0(Z\times X)\to\maL_{C(X)}(C(X)\otimes H)$ and $\pi': C_0(Z\times X)\to\maL_{C(X)}(C(X)\otimes H')$ as before. We also assume that $H$ and $H'$ are separable Hilbert spaces and strongly amplify the representations $(\pi, U, H)$ and $(\pi', U', H')$ as in \cite{BenameurRoyPPV} into
$$
({\what\pi}, {\what U}, {\what H})\text{ and } ({\what\pi}', {\what U}', {\what H}'), where
$$
\begin{itemize}
\item ${\what H}=H\otimes\ell^2(\Gamma)^\infty$ and ${\what H}'=H'\otimes\ell^2(\Gamma)^\infty$;
\item $\widehat{\pi}$ and  $\widehat{\pi}'$ are the (very ample) representations of $C_0(Z\times X)$ on $C(X)\otimes {\what H}$ and $C(X)\otimes {\what H}'$ respectively, obtained by tensoring $\pi$ and  $\pi'$ by the identity on $\ell^2(\Gamma)^\infty$.
\item $\widehat{U}$ and  $\widehat{U}'$ are the unitary representations of $\Gamma$ on ${\what H}$ and  ${\what H}'$ respectively, which are obtained by tensoring $U$ and $U'$ by right regular representation of $\Gamma$ on $\ell^2(\Gamma)$, and then by the identity on $\ell^2(\N)$.
\end{itemize}

\medskip
{It is then an obvious observation that the new representations $\what\pi$ and $\what\pi '$ are (fiberwise) ample. Recall the following theorem from \cite{BenameurRoyPPV}.}

\medskip

\begin{theorem}\cite{BenameurRoyPPV}
\label{BRPPV}\
{{Assume that the compact space $X$ is metric and finite dilmensional. Recall that $\pi$ and $\pi'$ are (fiberwise) faithful and non-degenerate representations of $C_0(Z)$ in $C(X)\otimes H$ and $C(X)\otimes H'$ respectively.}} Denote as well by  $\widehat{\pi}$ and $\widehat{\pi}'$ the trivially extended representations $\begin{pmatrix} \widehat{\pi} & 0\\
0 & 0\end{pmatrix}$
and 
$\begin{pmatrix} \widehat{\pi}' & 0\\
0 & 0\end{pmatrix}$
on $C(X)\otimes ({\what H}\oplus {\what H}')$ and $C(X)\otimes ({\what H}'\oplus {\what H})$ respectively. Then 
there exists a unitary operator
$$
V:C(X)\otimes ({\what H}\oplus {\what H}')\longrightarrow C(X)\otimes ({\what H}'\oplus {\what H})\text{ such that }
$$
\begin{enumerate}
\item $V$ is $\Gamma$-equivariant;
\item $V$ has finite propagation;
\item For any $g\in C_0(Z\times X)$, $V^*\widehat{\pi}'(g)V-\widehat{\pi}(g)\in\maK_{C(X)}(C(X)\otimes ({\what H}\oplus {\what H}'))$.
\end{enumerate}
\end{theorem}

\medskip

{\begin{remark}
In \cite{BenameurRoyPPV}, this theorem was more precisely stated under the stronger assumption that  the representations $\pi$ and $\pi'$ are (fiberwise) ample. It is worth pointing out  though that if $\pi$ is a faithful and non-degenerate representation, then $\pi^\infty=\pi\otimes\id_{\ell^2(\N)}$ is ample and unitarily equivalent to $(\pi^\infty)^\infty$ through a $\Gamma$-equivariant conjugation isomorphism.
\end{remark}} 

Notice that the unitary representations ${\what U}$ and ${\what U}'$ yield the unitary representations ${\what U}\oplus {\what U}'$ and ${\what U}'\oplus {\what U}$.

\medskip

Recall that we then have the corresponding representations 
\begin{multline*}
\what\pi\rtimes\Gamma:C_0(Z\times X)\rtimes\Gamma\longrightarrow \maL_{C^*(\maG)}\left(\left(C(X)\otimes ({\what H}\oplus {\what H}')\right)\rtimes\Gamma\right)
\mbox{  and  }\\
\what\pi'\rtimes\Gamma :C_0(Z\times X)\rtimes\Gamma\longrightarrow \maL_{C^*(\maG)}\left(\left(C(X)\otimes ({\what H}'\oplus {\what H})\right)\rtimes\Gamma\right)
\end{multline*}
which are also the trivial extensions, and are given  by the formulae
\begin{multline*}
(\what\pi\rtimes\Gamma)(g)(\xi)(x, \gamma)=\sum_{\mu\in\Gamma}\widehat{\pi}_x(g(\bullet, x, \mu))\widehat{U}_\mu\,\xi(\mu^{-1}x, \mu^{-1}\gamma)
\mbox{  and  similarly }\\
(\what\pi'\rtimes\Gamma)'(g)(\xi)(x, \gamma)=\sum_{\mu\in\Gamma}\widehat{\pi}'_x(g(\bullet, x, \mu))\widehat{U}'_\mu\, \xi(\mu^{-1}x, \mu^{-1}\gamma).
\end{multline*}
We also have the  trivial identifications
\begin{multline*}
(\what\pi\rtimes\Gamma)(e)((C(X)\otimes ({\what H}\oplus {\what H}'))\rtimes\Gamma)\simeq
(\what\pi\rtimes\Gamma)(e)((C(X)\otimes {\what H})\rtimes\Gamma)\\
\mbox{  and  }
(\what\pi'\rtimes\Gamma)(e)((C(X)\otimes ({\what H}'\oplus {\what H}))\rtimes\Gamma)\simeq
(\what\pi'\rtimes\Gamma)(e)((C(X)\otimes {\what H}')\rtimes\Gamma).
\end{multline*}
Recall that an element $\eta\in C_c(X\times\Gamma, {\what H})$ is $\Gamma$-invariant if
 $$
\eta (\gamma_0 x, \gamma_0 \gamma) = {\what U}_{\gamma_0} \eta (x, \gamma), \quad \forall (x, \gamma, \gamma_0)\in X\times \Gamma^2.
 $$
 Denoting as usual by $C(X\times\Gamma,  {\what H})^\Gamma$ the space of continuous $\Gamma$-invariant elements, it will be convenient to denote by 
$$
\maP: C_c(X\times\Gamma,  {\what H}) \rightarrow C(X\times\Gamma,  {\what H})^\Gamma \quad  \text{ and }\quad \maP' : C_c(X\times\Gamma,  {\what H}') \rightarrow C(X\times\Gamma,  {\what H}')^\Gamma
$$
 the well defined average linear map given by 
 $$
 \maP (\xi) (x, \gamma) := \sum_{\mu\in \Gamma} {\what \pi} (\mu c)\,  {\what U}_\mu \left(\xi (\mu^{-1}x, \mu^{-1}\gamma)\right), \text{ for }\xi\in C_c(X\times \Gamma, {\what H})
 $$
 and similarly for $\maP'$. That the ranges of $\maP$ and $\maP'$ are composed of $\Gamma$-invariant  elements is an obvious verification. Then we may rewrite some of the previous operators in terms of the average operators $\maP$ and $\maP'$, we have for instance on compactly supported elements
  $$
 ({\what \pi}\rtimes \Gamma) (e) = ({\what\pi} (c)\rtimes \Gamma) \circ \maP, \quad ({\what \pi}'\rtimes \Gamma) (e) = ({\what\pi}' (c)\rtimes \Gamma) \circ \maP'.
 $$


\medskip

\begin{lemma}\label{V-W}\
The unitary operator $V$ induces a well defined unitary operator 
$$
W: ({\what\pi}\rtimes\Gamma) (e) \left((C(X)\otimes {\what H})\rtimes \Gamma\right) \longrightarrow ({{\what\pi}'}\rtimes\Gamma) (e) \left((C(X)\otimes {\what H}')\rtimes \Gamma\right)
$$
obtained by restricting the adjointable operator $W_0:(C(X)\otimes {\what H})\rtimes \Gamma \rightarrow (C(X)\otimes {\what H}')\rtimes \Gamma$ induced  by 
$$
W_0:=\left(({\what \pi}'(c)\circ V)\rtimes \Gamma\right)\circ \maP.
$$
Moreover, $W^*=W^{-1}$ is also the restriction of the adjoint $W_0^*$, this latter being itself induced by  $W_0^*:=\left(({\what \pi}(c)\circ V^*)\rtimes \Gamma\right)\circ \maP'$.
\end{lemma}

\medskip
So, $W$ is the restriction to the range of the projection $({\what\pi}\rtimes\Gamma) (e)$ of the operator $W_0$ induced by
$$
W_0(\xi)(x, \gamma) := \widehat{\pi}'_x(c)V_x \sum_{\mu\in\Gamma}\widehat{\pi}_x(\mu c)\widehat{U}_\mu\xi(\mu^{-1}\gamma,\mu^{-1}x),  \text{ for }\xi\in C_c(X\times\Gamma, {\what H}).
$$
Notice that for $f\in C_c(Z\times X)$, the operator ${\what\pi} (f) \circ V$ is also  viewed as acting from $C(X)\otimes {\what H}$ to $C(X)\otimes {\what H'}$, which means that we are composing $V$ with the inclusion $P^*: {\what H}\hookrightarrow {\what H}\oplus {\what H}'$ where $P: {\what H}\oplus {\what H}'\to {\what H}$ is the projection. More precisely, the operator $P\circ {\what\pi} (f) \circ V\circ P^*$ is just  denoted ${\what\pi} (f) \circ V$ when no confusion can occur.

\begin{proof}\
Using the $\Gamma$-invariance of $V$ as well as the $\Gamma$-equivariance of the representations the following relations hold on  $C_c(X\times\Gamma, {\what H})$ 
$$
\maP\circ ({\what \pi}(c)\rtimes \Gamma)\circ \maP= \maP\quad \text{ and }\quad \maP'\circ  ({\what \pi}'(c) V\rtimes \Gamma)\circ \maP = (V\rtimes\Gamma)\circ \maP.
$$
{{More generally, the operators $\maP ({\what \pi}(c)\rtimes \Gamma)$ and $\maP'\circ  ({\what \pi}'(c) \rtimes \Gamma)$ act as the identity operators on  $\Gamma$-invariant elements}}. Now, the operator $W_0$ extends to an  adjointable operator with the adjoint given by the extension of the operator $
W_0^*=\left(({\what \pi}(c)\circ V^*)\rtimes \Gamma\right)\circ \maP'$.
Moreover,  using the above  relations  and the $\Gamma$-invariance of $V$ again, we can compute:
\begin{eqnarray*}
W_0\circ ({\what\pi}\rtimes\Gamma) (e) &= & \left(({\what \pi}'(c)\circ V)\rtimes \Gamma\right)\circ \maP\circ ({\what\pi} (c)\rtimes \Gamma)\circ \maP\\
& = &  \left(({\what \pi}'(c)\circ V)\rtimes \Gamma\right)\circ \maP\\
& =& \left({\what \pi}'(c)\rtimes \Gamma\right)\circ \maP'\circ \left({\what \pi}'(c)\rtimes\Gamma\right)\circ (V\rtimes\Gamma)\circ \maP\\
&=& ({\what \pi}'\rtimes\Gamma) (e) \circ W_0.
\end{eqnarray*}
Therefore, the operator $W_0$ sends the range of the projection $({\what\pi}\rtimes\Gamma) (e)$ into the range of the projection $({\what\pi}'\rtimes\Gamma) (e)$ and  the restriction $W$ is thus well defined,  no need to contract with the projections. In the same way the relation  $W_0^*\circ ({\what\pi}'\rtimes\Gamma) (e) = ({\what\pi}\rtimes\Gamma) (e) \circ W_0^*$ shows that $W^*$ can also be defined as the restriction of $W_0^*=\left(({\what \pi}(c)\circ V^*)\rtimes \Gamma\right)\circ \maP'$ to the range of the projection $({\what\pi}'\rtimes\Gamma) (e)$.

Finally, a direct computation gives
\begin{eqnarray*}
W_0^*\circ W_0 &=& \left({\what\pi} (c) V^* \rtimes\Gamma\right) \circ (\maP' \circ ({\what\pi}'(c)\rtimes\Gamma))\circ (V\rtimes\Gamma)\circ \maP\\
& = & ({\what\pi} (c)  \rtimes\Gamma)\circ (V^*V\rtimes\Gamma)\circ \maP\\
& = & ({\what\pi}(c)\rtimes\Gamma)\circ \maP\\
&=& ({\what\pi}\rtimes \Gamma) (e).
\end{eqnarray*}
and we get similarly $
W_0 \circ W_0^* =  ({\what\pi}'\rtimes \Gamma) (e)$.
Therefore, $W$ is an adjointable  unitary operator with the announced expression for its adjoint.

\end{proof}

\medskip

Using the unitary isomorphisms $\phi^{\what H}$ and $\phi^{{\what H}'}$  given in Lemma \ref{Smough2}, we may conjugate the unitary $W$ and deduce the  unitary (still denoted  $W$):
$$
W: \maE^{\what H}\longrightarrow \maE^{{\what H}'}.
$$

\medskip

\begin{theorem}\label{Functorial}
Conjugation by the unitary $W$ induces an isomorphism between the Roe $C^*$-algebras, i.e.
\begin{eqnarray*}
\Ad_W: D^*_\Gamma(\maE^{\what H}) &\overset{\simeq}\longrightarrow& D^*_\Gamma(\maE^{{\what H}'})\\
{\what T} &\longmapsto & W  \circ {\what T}\circ W^*.
\end{eqnarray*}
Moreover, the $K$-theory isomorphism induced by  $\Ad_W$ is independent  of the choice of the unitary operator $V$ given by  Theorem \ref{BRPPV}.
\end{theorem}

\medskip

\begin{proof}
Let ${\what T}\in D^*_\Gamma(\maE^{\what H})$ be a lift of some  operator $T$ from D$^*_\Gamma(X,Z,{\what H})$.
If $P:H'\oplus H\to H'$ denotes the projection, then the operator  $T'=VTV^*$, or rather $PV\begin{pmatrix} T & 0\\
0 & 0\end{pmatrix}V^*P^*$  belongs to $D^*_\Gamma(X,Z,{\what H}')$, see \cite{BenameurRoyPPV}. 
Let us show that that 
the defect operator
$$
\Ad_W((\what\pi\rtimes\Gamma)(e)(T\rtimes\Gamma)(\what\pi\rtimes\Gamma)(e))-(\what\pi'\rtimes\Gamma)(e)(T'\rtimes\Gamma)(\what\pi'\rtimes\Gamma)(e)
$$ 
is a compact operator, this will prove, using Proposition \ref{locally compact2}, that $\Ad_W({\what T})$ is  a lift of the operator $T'$. But this defect operator is completely determined by its action on the Hilbert $C^*(\maG)$-module
$$
 (\what\pi'\rtimes\Gamma)(e) \left((C(X)\otimes {\what H}')\rtimes\Gamma\right)\simeq \maE^{{\what H}'}.
$$ 
We then have $[T, {\what \pi}(c)]\rtimes \Gamma \sim 0$ where as before $\sim$ means equality modulo compact operators. Therefore, computing first on continuous compactly supported elements, we can eventually deduce
\begin{eqnarray*}
 W_0 (T\rtimes\Gamma)W_0^* & = & ({\what \pi}'(c)V\rtimes \Gamma) \maP (T{\what \pi}(c)\rtimes \Gamma) (V^*\rtimes \Gamma)\maP'\\
& \sim & ({\what \pi}'(c)V\rtimes \Gamma)\maP ({\what \pi}(c)\rtimes \Gamma) (TV^*\rtimes \Gamma)\maP'\\
& = & ({\what \pi}'(c)\rtimes \Gamma) (VTV^*\rtimes \Gamma)\maP'
\end{eqnarray*}
{{Recall for the last equality that the operator $\maP ({\what \pi}(c)\rtimes \Gamma)$ acts as the identity operator on any $\Gamma$-invariant element, and hence on  the range of $(TV^*\rtimes \Gamma)\maP'$.}} Recall also that $({\what\pi}'\rtimes\Gamma)(e)$ is the adjointable extension of $({\what \pi}'(c)\rtimes \Gamma)\maP'$, therefore
$$
({\what\pi}'\rtimes\Gamma)(e)W_0 (T\rtimes\Gamma)W_0^* ({\what\pi}'\rtimes\Gamma)(e) \sim  ({\what \pi}'(c)\rtimes \Gamma)\maP'({\what \pi}'(c)\rtimes \Gamma) (VTV^*\rtimes \Gamma)\maP'({\what \pi}'(c)\rtimes \Gamma)\maP'.
$$
{{But 
$$
\maP'({\what \pi}'(c)\rtimes \Gamma)\maP' = \maP'\quad \text{ and } \quad  \maP'({\what \pi}'(c)\rtimes \Gamma) (VTV^*\rtimes \Gamma)\maP' = (VTV^*\rtimes \Gamma)\maP'.
$$
Therefore, we deduce that 
$$
({\what\pi}'\rtimes\Gamma)(e)W_0 (T\rtimes\Gamma)W_0^* ({\what\pi}'\rtimes\Gamma)(e) \sim ({\what \pi}'(c)\rtimes \Gamma)(VTV^*\rtimes \Gamma)\maP'.
$$
Since $W$ is induced by $W_0$  from the range of $({\what\pi}\rtimes\Gamma)(e)$ and the range of $({\what\pi}'\rtimes\Gamma)(e)$, say $({\what\pi}'\rtimes\Gamma)(e)W_0 = W ({\what\pi}\rtimes\Gamma)(e)$ as operators between the Connes-Skandalis Hilbert modules, and hence we have}}
$$
W({\what\pi}\rtimes\Gamma)(e) (T\rtimes\Gamma) ({\what\pi}\rtimes\Gamma)(e)W^* \sim  ({\what \pi}'(c)\rtimes \Gamma)(VTV^*\rtimes \Gamma)\maP'.
$$
Computing similarly $(\what\pi'\rtimes\Gamma)(e)(T'\rtimes\Gamma)(\what\pi'\rtimes\Gamma)(e)$ we get
$$
({\what \pi}'(c)\rtimes \Gamma)\maP' (VTV^*\rtimes \Gamma) ({\what\pi}'(c)\rtimes \Gamma)\maP'.
$$
But $VTV^*{\what \pi}'(c)\sim {\what \pi}'(c)VTV^*$ since $V$ essentially intertwines the representations and since $T$ is pseudolocal. Therefore using the first item of Lemma \ref{rtimes} again together with the equality $\maP'   ({\what\pi}'(c)\rtimes \Gamma) (VTV^*\rtimes \Gamma)\maP'= (VTV^*\rtimes \Gamma)\maP'$, we deduce that :
$$
(\what\pi'\rtimes\Gamma)(e)(T'\rtimes\Gamma)(\what\pi'\rtimes\Gamma)(e) \sim ({\what \pi}'(c)\rtimes \Gamma) (VTV^*\rtimes \Gamma)\maP'.
$$

It remains to prove that $\Ad_W$ does not depend on the choice of $V$, but this is standard. Suppose that $V_1$ and $V_2$ are two unitaries as in Theorem \ref{BRPPV}, then  $W_1$ and $W_2$ are the induced unitaries between $(\pi\rtimes\Gamma)(e)((C(X)\otimes {\what H})\rtimes\Gamma)$ and $(\pi\rtimes\Gamma)(e')((C(X)\otimes {\what H}')\rtimes\Gamma)$.
The operators $W_iW^*_j$ ($i,j=1,2$) are clearly elements of the $C^*$-algebra $D^*_\Gamma(\maE^{{\what H}'})$.
For $T\in D^*_\Gamma(\maE^{\what H})$ the following relation holds in $M_2(D^*_\Gamma(\maE^{{\what H}'}))$:
$$M\begin{pmatrix} W_1TW_1^* & 0\\
0 & 0\end{pmatrix}M
=\begin{pmatrix} 0 & 0\\
0 & W_2TW_2^*\end{pmatrix} \; \text{ where }\; M=\begin{pmatrix} 0 & W_1W_2^*\\
W_2W_1^* & 0\end{pmatrix}.
$$
Hence we conclude that $(\Ad_{W_1})_*=(\Ad_{W_2})_*$ on $K$-theory.
\end{proof}

{{
\begin{remark}
In the previous proofs of Lemma \ref{V-W} and Theorem \ref{Functorial}, we have used the average operators $\maP$ and $\maP'$ only for simplicity. One can as well give direct proofs, see some computations in \cite{VictorThesis}.
\end{remark}
}
\ms

Theorem \ref{Functorial} allows to denote by 
$$
\Psi_{\pi, \pi'}^{Z} : K_* (D^*_\Gamma (\maE_Z^{\what H})) \longrightarrow K_* (D^*_\Gamma (\maE_{Z}^{{\what H}'})),
$$
the $\Z_2$-graded $K$-theory morphism induced by $Ad_W$ for any given faithful non-degenerate representations $\pi$ and $\pi'$. 
It is easy to check that the isomorphism $\Ad_W$ of Theorem \ref{Functorial} sends compacts to compacts and induces an isomorphism 
$$
\overline{\Ad}_W:Q^*_\Gamma(\maE_Z^{\what H}) \longrightarrow Q^*_\Gamma(\maE_Z^{{\what H}'}),
$$
such that the induced map on $K$-theory again does not depend on the choice of $V$.  Moreover, we have the following Proposition, recall  the isomorphism $\theta$ of Proposition \ref{equivalence2}.

\begin{proposition}\label{AdW}\
The following diagram commutes:
$$\xymatrix{
{Q^*_\Gamma(\maE_Z^{\what H})} \ar@{->}[r]^{\overline{\Ad}_W} \ar@{->}[d]^{{\theta^{-1}}}  & {Q^*_\Gamma(\maE_{Z}^{{\what H}'})} \ar@{->}[d]^{{\theta^{-1}}} 
\\
{Q^*_\Gamma(X,Z, {\what H})} \ar@{->}[r]^{\overline{\Ad}_V} & {Q^*_\Gamma(X,Z,{\what H}')}  
}$$
\end{proposition}

\begin{proof}
We only need to  check that for any  $T\in D^*_\Gamma(X,Z,{\what H})$, the operator
$$
(\pi'\rtimes\Gamma)(e)(\Ad_V(T)\rtimes\Gamma)(\pi'\rtimes\Gamma)(e)-\Ad_W((\pi\rtimes\Gamma)(e)(T\rtimes\Gamma)(\pi\rtimes\Gamma)(e))
$$
is a compact operator of the Hilbert $C^*(\maG)$-module $(C(X)\otimes {{\what H}'})\rtimes\Gamma$.
This is straightforward using the proof of Proposition \ref{locally compact2}, see also the proof of Theorem \ref{Functorial}.
\end{proof}

We end this section by gathering all the previous results and stating the resulting  analytic Higson-Roe sequence for the proper cocompact $\Gamma$-space $Z$, which corresponds to the fixed crossed product admissible completion. Let us briefly recall the Baum-Connes index map associated with the proper cocompact $\Gamma$-space $Z$.
The equivariant Kasparov group of the pair of $\Gamma$-algebras $(C_0(Z), C(X))$, say the group $KK^i_\Gamma (C_0(Z), C(X))$ has also been denoted $KK_\Gamma^i(Z, X)$ for short.  The  descent morphism $J_\Gamma$ was defined by Kasparov  in \cite{Kasparov} for the maximal completion, 
$$
J_\Gamma\, : \;  KK_\Gamma^i(Z, X) \longrightarrow KK_i (C_0(Z)\rtimes \Gamma, C(X)\rtimes_{\max} \Gamma).
$$ 
This descent morphism is defined for any pair of $\Gamma$-algebras, and its main property is the compatibility with the Kasparov product, i.e. for $\Gamma$-algebras $A, B$ and $C$, and  for any $x\in KK_i^\Gamma (A, B)$ and $y\in KK_j^\Gamma (B, C)$
$$
J_\Gamma (x\otimes_B y) = J_\Gamma (x) \otimes_{B\rtimes_{\max}\Gamma} J_\Gamma (y)\in KK_{i+j} (A\rtimes_{\max} \Gamma, C\rtimes_{\max} \Gamma).
$$
On the other hand, by choosing a cutoff function for the proper cocompact action of $\Gamma$ on $Z$, we consider again the associated Mishchenko idempotent $e\in C_c(Z)\rtimes {{\Gamma}}$ which hence represents a class  $[e]\in KK (\C, C_0(Z)\rtimes \Gamma)$ which does not depend on the choice of cutoff function. Composing $J_\Gamma$ with the Kasparov product on the left (over $C_0(Z)\rtimes \Gamma$) by $[e]$, we obtain the maximal Baum-Connes index map associated with $Z$:
$$
\mu^Z_{i,\maG}: KK_\Gamma^i (Z, X) \longrightarrow KK_i (\C, C_{\max}^*(\maG))\simeq K_i (C(X)\rtimes_{\max}\Gamma).
$$
Composition of this morphism with the natural morphism $K_i (C(X)\rtimes_{\max}\Gamma)\rightarrow K_i (C(X)\rtimes \Gamma)$ corresponding to our admissible crossed product yields the Baum-Connes index map for $Z$ and relative to the fixed admissible completion. Then we can state:

\begin{theorem}
For any proper cocompact metric space $Z$, there exists a  periodic  six-term exact sequence:
$$\xymatrix{
KK^0_{\Gamma} (Z, X) \ar@{->}[r]^{\mu^\maG_{0, Z}\hspace{0,3cm}}\ar@{<-}[d] &
K_0(C(X)\rtimes\Gamma)\ar@{->}[r]&
K_0(D^*_\Gamma (X; Z, \what{H}))\ar@{->}[d]\\
K_1(D^*_\Gamma (X; Z, \what{H}))\ar@{<-}[r]&
K_1(C(X)\rtimes\Gamma)\ar@{<-}[r]^{\hspace{0,3cm}\mu^\maG_{1, Z}}&
KK^1_{\Gamma} (Z, X) 
}$$
where $\mu^\maG_{i, Z}$ is the Baum-Connes index map associated with the proper cocompact $\Gamma$-space $Z$ with coefficients in the $\Gamma$-algebra $C(X)$. 
\end{theorem}

 We now investigate the functoriality  properties of this sequence and deduce our universal six-term exact sequence.

\section{The  universal structure group}

\subsection{Review of the BC assembly map}\label{BC-review}

We devote this  paragraph to a short review of the Baum-Connes  maps (BC map) for the groupoid $\maG=X\rtimes \Gamma$. The universal BC maps are group morphisms defined for $i=0, 1$,
$$
\mu_{i,\maG}: RK_i(\Gamma, X) \longrightarrow K_i (C(X)\rtimes\Gamma)=K_i (C^*(\maG)),
$$
which are defined for any admissible crossed product.  In fact, if we denote by $\mu^{\max}_{i,\maG}: RK_i (\Gamma, X)) \longrightarrow K_i (C_{\max}^*(\maG))$ the BC maps for the maximal completion, then the above BC-maps are compositions of $\mu^{\max}_{i,\maG}$ with the  the natural morphism $K_i(C_{\max}^*(\maG))\rightarrow K_i(C^* (\maG))$ induced by the morphism $C_{\max}^*(\maG)\to C^*(\maG)$. The reduced BC maps will be denoted $\mu^{r}_{i,\maG}: {{RK_i (\Gamma, X)}} \longrightarrow K_i (C_r^*(\maG))$, they are thus the composite maps of $\mu^{\max}_{i,\maG}$ with   the natural morphism $K_i(C_{\max}^*(\maG))\rightarrow K_i(C^*_r (\maG))$. In \cite{BaumConnes}, see also \cite{BaumConnesHigson}, the reduced BC-maps were conjectured to be isomorphisms. This is a special case of the celebrated Baum-Connes conjecture corresponding to the case of actions of countables discrete groups on compact spaces. While the conjecture is known to be true for a wide class of groups including for instance a-T-menable groups, especially amenable groups, its validity relied on the then expected exactness of the group, i.e. exactness of the reduced crossed product. The discovery of non-exact groups like the Gromov monster groups, enabled Higson-Lafforgue-Skandalis to  produce counter-example to the Baum-Connes conjecture for groupoids including transformation groupoids. This lack of exactness has recently been remedied by Baum-Guentner-Willett in \cite{BaumGuentnerWillett} by proposing a more approriate RHS of the Baum-Connes map. More precisely, they were able to replace the reduced crossed product by a more appropriate admissible crossed product, built as the ``minimal'' exact and Morita compatible admissible crossed product. They proved more precisely in \cite{BaumGuentnerWillett} that there always exists such a minimal crossed product functor associated with actions of  groups like our  $\Gamma$, for which all the counter-examples become confirmations of the new Baum-Connes conjecture, that we shall call here the rectified BC conjecture. We shall denote this crossed product by $\rtimes_{\ep}$ in reference to \cite{BaumGuentnerWillett}. This rectified conjecture can be stated as follows. Let us denote  $C^*_{\ep} (\maG)=C(X)\rtimes_{\ep} \Gamma$ the crossed product obtained using the Baum-Guentner-Willett (BGW) minimal crossed product.

\medskip

\begin{conjecture}[rectified BC conjecture]
For $i\in \Z_2$, the composite morphism 
$$
\mu^{BGW}_{i,\maG}: {{RK_i (\Gamma, X)}} \longrightarrow K_i (C_{\ep}^*(\maG)),
$$
of $\mu^{\max}_{i,\maG}$ with   the natural morphism $K_i(C_{\max}^*(\maG))\rightarrow K_i(C^*_{\ep} (\maG))$, is an isomorphism.
\end{conjecture}

\medskip
So, the right hand side is the usual  $K$-theory group of the $C^*$-algebra  $C_{\ep}^*(\maG)$, an analytical data associated with the groupoid $\maG$ which is expected to be computed precisely using the left hand side. 

\begin{remark}
The rectified conjecture is stated in \cite{BaumGuentnerWillett}  in the whole generality of actions of locally compact second-countable Hausdorff groups on $C^*$-algebras. But we prefer to concentrate here on our case of interest, say actions of countable discrete groups on compact spaces.
\end{remark}

\begin{remark}
If $\Gamma$ is a non-exact group, then amenable actions on compact spaces don't exist and in general $C_{\ep}^*(\maG)$ is not $KK$-equivalent to $C_r^*(\maG)$, and is even expected to sometimes be $KK$-equivalent to $C_{\max}^*(\maG)$! If $\Gamma$ is an exact group, then the rectified BC conjecture reduces to the BC conjecture. In this case, there exist compact spaces with amenable actions of $\Gamma$, however we want to apply our results even when our metric finite dimensional compact Hausdorff space $X$ is not  acted on amenably. 
\end{remark}


Let us recall one of the equivalent definitions of the left hand side of the (rectified) Baum-Connes map that will be used in the sequel, which  uses the classifying space ${\underline{E}}\Gamma$ for proper $\Gamma$-actions. It will also be denoted below $R K_{i, \Gamma} ({\underline{E}}\Gamma, X)$ since it corresponds as well to compactly supported $\Gamma$-equivariant $K$-homology of $ {\underline{E}}\Gamma$ with coefficients in the $\Gamma$-algebra $C(X)$. 
Given a  proper and cocompact $\Gamma$-subspace $Z$  of $ {\underline{E}}\Gamma$, we have the corresponding Baum-Connes index map 
$$
\mu^Z_{i,\maG}: KK_\Gamma^i (Z, X) \longrightarrow KK_i (\C, C ^*(\maG)) =  K_i (C(X)\rtimes \Gamma).
$$
The classical (i.e. reduced) Baum-Connes index map for $Z$ corresponds to the reduced crossed product (so the minimal admissible one), while the rectified Baum-Connes assembly map for $Z$ is the one associated with the BGW-crossed product completion, say the minimal exact and Morita compatible one. 

 If $\iota: Z\hookrightarrow Z'$ is a $\Gamma$-equivariant inclusion of a closed subspace of the proper cocompact space $Z'$, then $\iota$ defines the functoriality class $[\iota]\in KK_\Gamma (Z', Z)$, and hence a morphism $\iota_*$ which is Kasparov product on the left (over $C_0(Z)$) by the class $[\iota]$. Notice that if we use as cutoff function on $Z$ the restriction of a cut-off function on $Z'$, then it is easy to check that $[e']\otimes_{C_0(Z')\rtimes\Gamma} J_\Gamma [\iota]= [e]$ where $e'$ is the Mishchenko idempotent for $Z'$. Therefore, for any $x\in KK^i_\Gamma (Z, X)$, we can write
$$
\mu^Z_{i,\maG} (x) = [e]\otimes_{C_0(Z)\rtimes\Gamma} J_\Gamma (x) = \left([e']\otimes_{C_0(Z')\rtimes\Gamma} J_\Gamma [\iota]\right)\otimes_{C_0(Z)\rtimes\Gamma} J_\Gamma (x).
$$
Associativity of the Kasparov product gives 
$$
\mu^Z_{i,\maG} (x)  = [e']\otimes_{C_0(Z')\rtimes\Gamma} \left(J_\Gamma [\iota] \otimes_{C_0(Z)\rtimes\Gamma} J_\Gamma (x)\right) = [e']\otimes_{C_0(Z')\rtimes\Gamma} J_\Gamma \left( [\iota]\otimes_{C_0(Z)} x\right).
$$
Therefore, we have the compatibility property
$$
\mu^Z_{i,\maG} (x) = \mu^{Z'}_{i,\maG} (\iota_*x). 
$$
The same compatibility then holds for any compatible crossed product completion. Let us denote as usual by ${\underline{E}}\Gamma$ a classifying space for proper actions of $\Gamma$. Then the above compatibility property with respect to inclusions $\iota: Z\hookrightarrow Z'$ allows to assemble all the locally compact proper and cocompact subspaces $Z$ of  ${\underline{E}}\Gamma$ and to state the following definition.

We shall denote by $RK_{i, \Gamma} ({\underline{E}}\Gamma, X)$ the inductive limit 
$$
RK_{i, \Gamma} ({\underline{E}}\Gamma, X) := \varinjlim_{Z\subset {\underline{E}}\Gamma} KK^i_\Gamma(Z,X),
$$
 where $Z$ runs over the proper cocompact $\Gamma$-subspaces of ${\underline{E}}\Gamma$, and the inductive limit is taken with respect to the system of maps $\iota_*$ corresponding to inclusions as above.

\begin{definition}\
 The rectified BC map for the groupoid $\maG=X\rtimes\Gamma$ is the resulting inductive limit morphism 
$$
\mu_{i, \maG} = \varinjlim_{Z\subset {\underline{E}}\Gamma} \mu_{i, \maG}^Z, \text{ so for instance }
\mu^{BGW}_{i, \maG}  \, : \; RK_{i, \Gamma} ({\underline{E}}\Gamma, X) \longrightarrow K_i(C(X)\rtimes_{\ep}\Gamma).
$$
\end{definition}

In the sequel we shall sometimes denote $RK_{i, \Gamma} ({\underline{E}}\Gamma, X)$ by $RK_{i, \maG} ({\underline{E}}\maG)$ when no confusion can occur. 

\medskip

\subsection{{{Independence}} of the non-degenerate representation}

To specifiy the space $Z$, we shall now  denote the Hilbert module   $\maE^{\what H}$ by $\maE_Z^{\what H}$. 
Theorem \ref{Functorial} shows that for any {(fiberwise) faithful non-degenerate} representations $\pi$ and $\pi'$ of $C_0(Z)$ in $C(X)\otimes H$ and $C(X)\otimes H'$ respectively, there exists a $\Z_2$-graded $K$-theory isomorphism
$$
\Psi^Z_{\pi, \pi'} : K_* (D^*_\Gamma (\maE_Z^{\what H})) \longrightarrow K_* (D^*_\Gamma (\maE_Z^{{\what H}'})),
$$
and this allowed us to prove that the isomorphism class of $K_* (D^*_\Gamma (\maE_Z^{\what H}))$ is  independent of the choice of the faithful non-degenerate representation $\pi$.

Let us explain now the functoriality morphism that one gets using  Theorem \ref{Functorial}. Let $\iota:Z\hookrightarrow Z'$ be a $\Gamma$-inclusion  of the proper cocompact locally compact $\Gamma$-spaces, so with $Z$ closed in $Z'$, and let $\pi:C_0(Z)\to \maL(H)$ and $\pi':C_0(Z')\to \maL(H')$ be two non-degenerate faithful representations. 
Thanks to Theorem \ref{Functorial},  for each of the spaces $Z$ and $Z'$ we have canonical isomorphism classes of the $K$-theory groups of their dual algebras. So the functoriality morphism will allow us to pass from a (canonical isomorphism class of a) $K$-group for $Z$ to a (canonical isomorphism class of a) $K$-group for $Z'$, i.e. it will be compatible with the isomorphisms  $\Psi_{\pi, \pi'}^Z$ of Theorem \ref{Functorial}. More precisely, we construct a $\Z_2$-graded $K$-theory  morphism 
$$
\iota_{\pi, \pi'}^{Z, Z'} : K_* (D^*_\Gamma (\maE_Z^{\what H})) \longrightarrow K_* (D^*_\Gamma (\maE_{Z'}^{{\what H}'})),
$$
with the allowed functoriality properties.  

Using the restriction morphism $\iota^*: C_0(Z')\to C_0(Z)$, the representation $\pi$ can be pushed forward to give the non-degenerate representation $\iota_*\pi$ in the same Hilbert space $H$, that we rather denote by $\iota_*H$. so, $(\iota_*\pi, \iota_*H)$ will be the same Hilbert space $H$ with the representation $\iota_*\pi=\pi\circ \iota^*$.  Clearly, $\iota_*\pi$  is a (non-degenerate) representation in $\maL_{C(X)} (C(X)\otimes H)$ which is no more faithfull unless $Z'=Z$. Now every operator in $D^*_\Gamma (\maE_Z^{\what H})$ is also an element of the dual algebra  $D^*_\Gamma (\maE_{Z'}^{\what {\iota_*H}})$ associated with the representation $\iota_*\pi$. Therefore, composing with  the  inclusion homomorphism 
$$
D^*_\Gamma (\maE_{Z'}^{\what {\iota_*H}}) \hookrightarrow D^*_\Gamma (\maE_{Z'}^{\what {\iota_*H}\oplus {\what H'}}) \text{ given by } T\longmapsto \left(\begin{array}{cc} T & 0 \\ 0 & 0\end{array}\right),
$$
we get a $\Z_2$-graded morphism between the $K$-theory groups:
$$
j_{\pi, \pi'}^{Z, Z'} : K_* (D^*_\Gamma (\maE_Z^{\what H})) \longrightarrow K_* (D^*_\Gamma (\maE_{Z'}^{\what {\iota_*H}\oplus {\what H'}}) ) = K_* (D^*_\Gamma (\maE_{Z'}^{\what {\iota_*H\oplus H'}})),
$$
Since $\pi'$ is a faithful non-degenerate representation of $C_0(Z')$, the direct sum representation $\iota_*\pi\oplus \pi'$ is also a faithful and non-degenerate representation. Therefore, Theorem \ref{Functorial} provides the $\Z_2$-graded group isomorphism 
$$
\Psi^{Z'}_{\iota_*\pi\oplus \pi', \pi'} : K_* (D^*_\Gamma (\maE_{Z'}^{\what {\iota_*H\oplus H'}})) \longrightarrow K_* (D^*_\Gamma (\maE_{Z'}^{{\what H}'})),
$$
which only depends on the representations $\pi$ and $\pi'$. 

\begin{definition}\
Let $\iota:Z\hookrightarrow Z'$ be a $\Gamma$-inclusion  of the proper cocompact locally compact $\Gamma$-spaces, so with $Z$ closed in $Z'$, and let $\pi:C_0(Z)\to \maL_{C(X)}(C(X)\otimes H)$ and $\pi':C_0(Z')\to \maL_{C(X)}(C(X)\otimes H')$ be two non-degenerate faithful representations. The functoriality morphism $\iota_{\pi, \pi'}^{Z, Z'}$ is the $\Z_2$-graded group morphism 
$$
\iota_{\pi, \pi'}^{Z, Z'} : K_* (D^*_\Gamma (\maE_Z^{\what H})) \longrightarrow K_* (D^*_\Gamma (\maE_{Z'}^{{\what H}'})),
$$
defined as the composite morphism 
$$
\iota_{\pi, \pi'}^{Z, Z'} := \Psi^{Z'}_{\iota_*\pi\oplus \pi', \pi'}\circ j_{\pi, \pi'}^{Z, Z'}.
$$
\end{definition}

\medskip

We deduce that, by definition and using again Theorem \ref{Functorial}, the functoriality morphism $\iota_{\pi, \pi'}^{Z, Z'}$ only depends, up to a conjugation isomorphism, on the $\Gamma$-pair $Z\hookrightarrow Z'$ of  proper cocompact $\Gamma$-spaces. Moreover, if $ Z'\hookrightarrow Z''$ is another $\Gamma$-inclusion, then one proves again that $
{\iota}_{\pi', \pi''}^{Z', Z''} \circ \iota_{\pi, \pi'}^{Z, Z'} = {\iota}_{\pi, \pi''}^{Z, Z''}$.

Notice that applying this property to the identity maps $Z\to Z$ and $Z'\to Z'$ and with two other  faithful non-degenerate representation $\pi_1: C_0(Z)\to \maL_{C(X)}(C(X)\otimes H_1)$ and $\pi'_1: C_0(Z')\to \maL_{C(X)}(C(X)\otimes H'_1)$, we also obtain the commutativity of the following square for $i\in \Z_2$
$$\xymatrix{
{K_i (D^*_\Gamma (\maE_Z^{\what H}))} \ar@{->}[r]^{\Psi_{\pi, \pi_1}^Z} \ar@{->}[d]^{\iota^{Z, Z'}_{\pi, \pi'}}
& {K_i (D^*_\Gamma (\maE_Z^{\what {H_1}}))}\ar@{->}[d]^{\iota^{Z, Z'}_{\pi_1, \pi'_1}} \\ 
{K_i(D^*_\Gamma (\maE_{Z'}^{\what {H'}}))}\ar@{->}[r]^{\Psi^{Z'}_{\pi', \pi'_1}} & {K_i(D^*_\Gamma (\maE_{Z'}^{{\what {H'_1}}}))}
}$$
%

According to the previous remarks,  we shall  denote unambiguisly $\iota_{\pi, \pi'}^{Z, Z'}$  by $\iota^{Z, Z'}$ or $\iota^{Z\subset Z'}$.  

%

\medskip

Now, again passing to the quotient $C^*$-algebras, one obtains  that the $\Gamma$-inclusion $\iota: Z\hookrightarrow Z'$, and the representations $\pi$ and $\pi'$,  yield a $\Z_2$-graded funtoriality isomorphism still denoted
$$
\iota^{Z, Z'}=\iota_{\pi, \pi'}^{Z, Z'}   : K_* (Q^*_\Gamma (\maE_Z^{\what H})) \longrightarrow K_* (Q^*_\Gamma (\maE_{Z'}^{{\what H}'})).
$$
Recall the isomorphism $\theta$ of {{Proposition \ref{equivalence2}}}. We can deduce from the very definitions of the functoriality morphisms the following

\begin{proposition}\label{QuotientCompa}\
Fix the faithful non-degenerate representations $\pi$ and $\pi'$ as above of $C_0(Z)$ and $C_0(Z')$, and denote for simplicity by $\iota_*: K_*(Q^*_\Gamma (X, Z, {\what H})) \rightarrow K_*(Q^*_\Gamma (X, Z', {\what H}'))$ the funtoriality $K$-theory morphism defined in \cite{BenameurRoyII} and also by $\iota_* : K_* (Q^*_\Gamma (\maE_Z^{\what H})) \rightarrow K_* (Q^*_\Gamma (\maE_{Z'}^{{\what H}'}))$ the above functoriality morphism. Then the following diagram commutes
$$
\xymatrix{
{K_i(Q^*_\Gamma(\maE_Z^{\what H}))} \ar@{->}[r]^{\iota_*} \ar@{->}[d]^{{\theta_*^{-1}}}
& {K_i(Q^*_\Gamma(\maE_{Z'}^{{\what H}'}))}\ar@{->}[d]^{{\theta_*^{-1}}} \\ 
{K_i(Q^*_\Gamma (X, Z, {\what H}))} \ar@{->}[r]^{\iota_*} & {{K_*(Q^*_\Gamma (X, Z', {\what H}'))}}
}
$$
\end{proposition}

\begin{proof}
The construction of the morphism $\iota_*: K_*(Q^*_\Gamma (X, Z, {\what H})) \rightarrow K_*(Q^*_\Gamma (X, Z', {\what H}'))$ follows the same steps, and we only need to check the compatibility of the isomorphisms $\theta_*$ with the isomorphisms $\Psi_{\pi, \pi'}^{Z, Z'}$ which are defined using ${\overline{Ad_W}}$ for the algebras $Q^*_\Gamma (\maE_Z^{\what H})$ and ${\overline{Ad_V}}$ for the algebras $Q^*_\Gamma (X, Z, {\what H})$. Hence, the proof is complete using Proposition \ref{AdW} and its proof. 
\end{proof}

 We can deduce the following

\begin{corollary}\label{Paschke}\
There exist  natural  Paschke   group isomorphisms
$$
\maP^H_{i, Z}\, : \, K_i(Q^*_\Gamma(\maE_Z^{\what H}))\stackrel{\simeq}{\longrightarrow} KK^{i+1}_\Gamma(Z,X), \quad\text{ for }i\in \Z_2.
$$
Moreover, these isomorphisms are compatible with the functoriality morphisms  $\iota^{Z, Z'}$.
\end{corollary}


\begin{proof}

In \cite{BenameurRoyII}, the  Paschke map is defined by the usual formula and it is proved there that it gives group isomorphisms
$$
K_i(Q^*_\Gamma(X,Z,{\what H}))\stackrel{\simeq}{\longrightarrow} KK^{i+1}_\Gamma(Z,X), \text{ for }i\in \Z_2.
$$
Hence, it suffices to define $\maP^H_{i, Z}$ by composing {{the Paschke isomorphism with the isomorphism $\theta_*^{-1}$}}. Recall on the other hand that the following diagram commutes, for  $\iota:Z\hookrightarrow Z'$ as above (see {{Proposition 5.5 in \cite{BenameurRoyII}}})
$$\xymatrix{
{K_i(Q^*_\Gamma(X,Z,{\what H}))} \ar@{->}[r]^{\iota_*} \ar@{->}[d]^{\simeq}
& {K_i(Q^*_\Gamma(X,Z',{\what H}'))}\ar@{->}[d]^{\simeq} \\ 
KK^{i+1}_\Gamma(Z,X)\ar@{->}[r]^{\iota_*} & KK^{i+1}_\Gamma(Z',X)
}$$
Propositions  \ref{QuotientCompa} then allows to conclude. More precisely, the following diagram commutes
$$\xymatrix{
{K_i(Q^*_\Gamma(\maE_Z^{\what H}))} \ar@{->}[r]^{\iota^{Z, Z'}} \ar@{->}[d]^{\maP^H_{i, Z}}
& {K_i(Q^*_\Gamma(\maE_{Z'}^{{\what H}'}))}\ar@{->}[d]^{\maP^{H'}_{i, Z'}} \\ 
KK^{i+1}_\Gamma(Z,X)\ar@{->}[r]^{\iota_*} & KK^{i+1}_\Gamma(Z',X)
}$$


\end{proof}

\medskip

\subsection{The universal  HR sequence}

 Let us get back now to the  Higson-Roe sequences and to the dual algebras associated with our groupoid $\maG$ and with any admissible crossed product functor $\rtimes$ for actions of  $\Gamma$. See again  \cite{BaumGuentnerWillett}.

 Recall that $\underline{E}\Gamma$ is a  classifying space for proper $\Gamma$-actions. Let $Z$ be a  $\Gamma$-subspace  of $\underline{E}\Gamma$ which is cocompact, then choosing a (fiberwise) faithful and non-degenerate $\Gamma$-equivariant representation $\pi_Z: C_0(Z) \rightarrow \maL_{C(X)} (C(X)\otimes H)$ in some Hilbert space $H$, we defined the $\Z_2$-graded group $K_*(D^*_\Gamma (\maE^{\what H}_Z)$, and any two choices give isomorphic groups. We also proved that the functoriality morphisms 
  $$
\iota^{Z, Z'}: K_i(D^*_\Gamma (\maE_Z^{\what {H}}))\longrightarrow K_i(D^*_\Gamma (\maE_{Z'}^{{\what {H'}}})).
$$
 corresponding to inclusions $Z\hookrightarrow Z'$ are compatible with the isomorphisms corresponding to different choices of representations. Moreover,  if $\iota':Z'\hookrightarrow Z''$ is another $\Gamma$-equivariant inclusion, then $
\iota^{Z', Z''}\circ \iota^{Z, Z'} = \iota^{Z, Z''}$.
Since no confusion can occur we drop the the superscript $\bullet^{\what {H}}$ and simply denote by $K_i(D^*_\Gamma (\maE_Z))$ the groups $K_i(D^*_\Gamma (\maE_Z^{\what {H}}))$. 
The next lemma is  standard and will be  used in the proof of Theorem \ref{ES2}.  

\begin{lemma}
\label{full}
The module $\maE_Z^{\what H}$ is a full Hilbert $C^*(\maG)$-module. In particular, the $C^*$-algebra $\maK_{C^*(\maG)}(\maE^{\what H}_Z)$ is Morita equivalent to the $C^*$-algebra $C^*(\maG)$ of the groupoid $\maG=X\rtimes\Gamma$. 
Hence, we have Morita group isomorphisms $\maM^Z: K_i(C^*(\maG)) \rightarrow K_i(\maK_{C^*(\maG)}(\maE^{\what H}_Z))$.  

Moreover, if $\iota:Z\hookrightarrow Z'$ is a $\Gamma$-inclusion as before, then  $\iota_*\circ \maM^Z = \maM^{Z'}$.
\end{lemma}

\begin{proof}\ 
We reproduce the standard proof in our case, see for instance \cite{ConnesSurvey} where the morphism $C^*(\maG)\to \maK_{C^*(\maG)}(\maE_Z^{\what H})$ is constructed, and also \cite{ConnesSkandalis} for the $KK$-version of the Morita quasi-trivial extensions. 
Let $f\in C_c(X\times\Gamma)$ be any continuous compactly supported function.
By using the fiberwise non-degeneracy of the representation $\pi$ plus the strong continuity of  operators from $\maL_{C(X)} (C(X)\otimes H)$, we can find $\pphi\in\pi(C_c(Z\times X))(C(X)\otimes H)$ such that
for any $x\in X$, we have $\Vert\pphi(x)\Vert_H=1$. Indeed,  for any $x\in X$ we can find a finite collection of elements   $\zeta^i\in H$ and $\psi^i\in C_c(Z)$, depending on $x$  such that  $\Vert \sum_i \pi_x(\psi^i) (\zeta^i) \Vert = 1$. Now for any $i$, the map $x'\mapsto  \pi_{x'}(\psi^i) (\zeta^i)$ is norm continuous, hence we can find a small open neighborhood  of $x$ over which  we have $\Vert \sum \pi_{x'} (\psi^i) (\zeta^i) \Vert \geq  3/4$. Compacity of $X$ allows to use  a finite partition of unity and to deduce the existence of a finite collection of functions $(\psi^\alpha)_\alpha$ in $C_c(Z\times X)$ and vectors $(\zeta^\alpha)_\alpha$ in $H$ such that for any $x\in X$,  $\Vert \sum_\alpha \pi_x(\psi^\alpha) (\zeta^\alpha) \Vert \geq 3/4$. Normalizing gives us eventually the element $\pphi$. 

 Let now  $h\in C_c(\Gamma\times\N)$ be the characteristic function of $(1_\Gamma, 0)$ and consider $\xi=\pphi\otimes h\in C(X)\otimes{\what H}$. Then clearly $\xi \in {\what \pi}(C_c(Z\times X))( C(X,{\what H}))$, and we check below  that
$$
f=\underset{\mu\in\Gamma}\sum\langle\xi,\eta_\mu\rangle\text{ where }\eta_\mu(x)=f(\mu^{-1} x,\mu^{-1})\, U_\mu\pphi(\mu^{-1}x)\otimes V_\mu h,
$$
{{where only in this proof $\mu\mapsto V_\mu$ is the tensor product of the regular representation by the identity of $\ell^2\N$}}. This will end the proof. Notice indeed that  $\eta_\mu=0$ except for a finite number of elements $\mu$ of $\Gamma$ ($f$ being compactly supported), and that each $\eta_\mu\in {\what \pi}(C_c(Z\times X))( C(X,{\what H}))$ as well. Hence,  $\langle {\what \pi}(C_c(Z\times X))( C(X,{\what H})),{\what \pi}(C_c(Z\times X))( C(X,{\what H}))\rangle$ contains $C_c(X\times\Gamma)$ and is therefore a dense subspace of $C^*(\maG)$ as allowed. Now, let us check the above relation $\underset{\mu\in\Gamma}\sum\langle\xi,\eta_\mu\rangle = f$. We compute
\begin{eqnarray*}
\sum_{\mu\in\Gamma}\langle\xi,\eta_\mu\rangle(x,\gamma) &=& 
\sum_{\mu\in\Gamma}\langle\xi(x),(U_\gamma\otimes V_\gamma)\eta_\mu(\gamma^{-1} x)\rangle_{{\what H}} \\
&=& \sum_{\mu\in\Gamma}\sum_{\beta\in\Gamma}\sum_{n\in\N}\langle
\pphi(x)h(\beta,n),f(\mu^{-1}\gamma^{-1} x,\mu^{-1})U_{\gamma\mu}\pphi(\mu^{-1}\gamma^{-1} x)(V_{\gamma\mu} h)(\beta,n) \rangle_H\\
&=& \sum_{\mu\in\Gamma}f(\mu^{-1}\gamma^{-1} x,\mu^{-1})h(\gamma\mu,0)
\langle\pphi(x),U_{ \gamma\mu}\pphi(\mu^{-1}\gamma^{-1}x)\rangle_H\\
&=&f(x,\gamma) \langle\pphi(x),\pphi(x)\rangle_H=f(x,\gamma).
\end{eqnarray*}
Therefore, $\maE_Z^{\what H}$ is a full Hilbert module as claimed. 

The module $\maE_Z^{\what H}$ being a full Hilbert $C^*(\maG)$-module, the $C^*$-algebra  $\maK_{C^*(\maG)}(\maE_Z^{\what H})$ is canonically Morita equivalent to $C^*(\maG)$. Moreover, an inspection of the Morita map $C^*(\maG)\to \maK_{C^*(\maG)}(\maE_Z^{\what H})$ which induces the Morita isomorphism $\maM^Z = K_i(C^*(\maG))\to K_i(\maK_{C^*(\maG)}(\maE_Z^{\what H}))$ shows that we have 
$$
\iota_*\circ \maM^Z = \maM^{Z'}: K_i(C^*(\maG))\longrightarrow K_i(\maK_{C^*(\maG)}(\maE_{Z'}^{{\what H}'})).
$$
Recall that the functoriality map $D^*_\Gamma (\maE_Z^{\what H})\to D^*_\Gamma (\maE_{Z'}^{{\what H}'})$ defined in the previous section sends compacts to compacts and induces a group morphism $\iota_*:K_i(\maK_{C^*(\maG)}(\maE_Z^{\what H})) \rightarrow K_i(\maK_{C^*(\maG)}(\maE_{Z'}^{{\what H}'}))$ which does not depend on the choices.

\end{proof}

\medskip

We can hence deduce the following

\begin{theorem}\label{ES2}\
Let $\iota:Z\to Z'$ be an inclusion of proper cocompact  $\Gamma$ spaces as above. Then the induced group morphism $\iota_*: K_i(D^*_\Gamma(\maE_Z))\rightarrow K_i(D^*_\Gamma(\maE_{Z'}))$ fits in a (functoriality) morphism between the  Higson-Roe six-term exact sequences associated with $Z$ and $Z'$ respectively. 
More precisely,  we have a commutative diagram for $i\in \Z_2$:
{\small{$$\xymatrix{
\dots \ar@{->}[r]^{\hspace{-0,25cm}\mu_{i,\maG}^Z}&
{K_i(C^*(\maG))}\ar@{->}[r] \ar@{->}[d]^{\id} &
K_i(D^*_\Gamma(\maE_Z))\ar@{->}[r] \ar@{->}[d]^{\iota_*} &
{KK^{i+1}_\Gamma (Z, X)} \ar@{->}[d]^{\iota_*} \ar@{->}[r]^{\mu_{i+1,\maG}^Z} &
{K_{i+1}(C^*(\maG))}\ar@{->}[r] \ar@{->}[d]^{\id} &
\dots
\\
\dots\ar@{->}[r]^{\hspace{-0,25cm}\mu_{i,\maG}^{Z'}}&
 K_i(C^*(\maG))\ar@{->}[r] &
K_i(D^*_\Gamma(\maE_{Z'}))\ar@{->}[r] &
{K_\Gamma^{i+1}(Z', X)}\ar@{->}[r]^{\mu_{i+1,\maG}^{Z'}} &
{K_{i+1}(C^*(\maG))}\ar@{->}[r] &
\dots
}$$}}
\end{theorem}

\medskip

\begin{proof}\ 
We have constructed the functoriality morphisms $\iota_*$ between the $K$-theories of $D^*_\Gamma(\maE_Z^{\what H})$ and those of $D^*_\Gamma(\maE_{Z'}^{\what H})$ which induce morphisms, still denoted $\iota_*$, between the $K$-theories of the quotient algebras  $Q^*_\Gamma(\maE_Z^{\what H})$ and  $Q^*_\Gamma(\maE_{Z'}^{\what H})$. Therefore, we have a commutative diagram
$${\tiny{\xymatrix{
\dots \ar@{->}[r]&
K_i(\maK_{C^*(\maG)}(\maE_Z^{\what H}))\ar@{->}[r] \ar@{->}[d]^{\iota_*} &
{K_i(D^*_\Gamma(\maE_Z^{\what H}))}\ar@{->}[r] \ar@{->}[d]^{\iota_*} &
{K_i(Q^*_\Gamma(\maE_Z^{\what H}))} \ar@{->}[d]^{\iota_*} \ar@{->}[r] &
K_{i+1}(\maK_{C^*(\maG)}(\maE_Z^{\what H}))\ar@{->}[r] \ar@{->}[d]^{\iota_*} &
\dots
\\
\dots\ar@{->}[r]&
 K_i(\maK_{C^*(\maG)}(\maE_{Z'}^{{\what H}'}))\ar@{->}[r] &
{K_i(D^*_\Gamma(\maE_{Z'}^{{\what H}'}))}\ar@{->}[r] &
{K_i(Q^*_\Gamma(\maE_{Z'}^{{\what H}'}))}\ar@{->}[r] &
K_{i+1}(\maK_{C^*(\maG)}(\maE_{Z'}^{{\what H}'}))\ar@{->}[r] &
\dots}}
}$$
Using Corollary \ref{Paschke},   we may intertwine by the Paschke isomorphisms and replace $K_i(Q^*_\Gamma(\maE_Z^{\what H}))$ by $KK^{i+1}_\Gamma (Z, X)$, and similarly $K_i(Q^*_\Gamma(\maE_{Z'}^{\what H}))$ by $KK^{i+1}_\Gamma (Z', X)$. Also, Lemma \ref{full} allows to intertwine using the Morita isomorphisms, and hence replace in this diagram $K_i(\maK_{C^*(\maG)}(\maE_Z^{\what H}))$ and $K_i(\maK_{C^*(\maG)}(\maE_{Z'}^{\what H}))$ by $K_i(C^*(\maG))$ with $\iota_*$ becoming the identity map. More precisely, Lemma \ref{full} shows that we have isomorphisms between the $K$-theory groups
$$
K_i(\maK_{C^*(\maG)}(\maE_Z^{\what H})) \stackrel{\simeq}{\longrightarrow} K_i(C^*(\maG)) \stackrel{\simeq}{\longleftarrow} K_i(\maK_{C^*(\maG)}(\maE_{Z'}^{{\what H}'})), 
$$
so that  the morphism $
\iota_*: K_i(\maK_{C^*(\maG)}(\maE_Z^{\what H}))\longrightarrow K_i(\maK_{C^*(\maG)}(\maE_{Z'}^{{\what H}'}))$
coincides with the identity isomorphism of $K_i(C^*(\maG))$. 

On the other hand, 
Now, the compatibility of the Baum-Connes maps associated with $Z$ and $Z'$, reviewed in the previous subsection, say the equality
$$
\mu_{i, \maG}^{Z'} \circ \iota_* = \mu_{i, \maG}^Z \, : \; KK^i_\Gamma (Z, X) \longrightarrow K_i (C^*(\maG)). 
$$
completes the proof.

\end{proof}

\begin{definition}
The universal  structure groups associated with the transformation groupoid $\maG=X\rtimes \Gamma$ and with the admissible crossed product $\rtimes\Gamma$ is the inductive limit group 
$$
\maS_i (\Gamma, X) :=  \varinjlim_{Z\subset {\underline{E}}\Gamma} K_{i+1} (D^*_\Gamma (\maE_Z)),
$$
where the inductive limit is taken with respect to the associated directed system of morphisms $\iota^{Z, Z'}$.
\end{definition}

So, our group $\maS_i (\Gamma, X)$ is defined using the choice of the collection of faithful non-degenerate representations, however, two such choices give canonically isomorphic groups.
Gathering the previous compatibility results we can state the main result of the present paper, i.e.  we obtain the  six-term exact sequence for the transformation groupoid $X\rtimes\Gamma$.

\begin{theorem}\label{HR-sequence}\
There exists a    Higson-Roe exact sequence associated with the transformation groupoid $\maG=X\rtimes\Gamma$:
$$\xymatrix{
RK_{0, \Gamma} ({\underline{E}}\Gamma, X) \ar@{->}[r]^{\mu_{0, \maG}\hspace{0,3cm}}\ar@{<-}[d] &
K_0(C(X)\rtimes\Gamma)\ar@{->}[r]&
\maS_1(\Gamma, X)\ar@{->}[d]\\
\maS_0(\Gamma, X)\ar@{<-}[r]&
K_1(C(X)\rtimes\Gamma)\ar@{<-}[r]^{\hspace{0,3cm}\mu_{1, \maG}}&
RK_{1, \Gamma} ({\underline{E}}\Gamma, X) 
}$$
\end{theorem}

\begin{proof}\
We apply  Theorem \ref{ES2}, and notice that the exactness of the six-term exact sequences of that theorem for  given $Z$ and $Z'$, together with the compatibility with respect to the inclusion $\iota:Z\hookrightarrow Z'$ as stated there, allows to deduce that all the maps in the universal sequence of Theorem \ref{HR-sequence} are well defined by passing to the direct limits. Moreover, it is easy to check that exactness is preserved, and the proof is now complete.
\end{proof}

When  we use the BGW crossed product then we shall denote the structure group by $\maS^{BGW}_i (\Gamma, X)$ or also by $\maS^{BGW}_i (\maG)$. The similar notation is used for the maximal and reduced groups, say respectively $\maS^{\max}_i (\Gamma, X)$ and $\maS^{r}_i (\Gamma, X)$. An immediate corollary of Theorem \ref{HR-sequence} is 

\begin{corollary}
The group $\Gamma$ satisfies the  rectified Baum-Connes conjecture with coefficients in $C(X)$ if and only if the structure groups $\maS^{BGW}_i(X\rtimes\Gamma)$ are trivial. 
\end{corollary}

Therefore, we expect that any invariants that could be extracted from the group $\maS^{BGW}_i(X\rtimes\Gamma)$ through a group homomorphism $\maS^{BGW}_i(X\rtimes\Gamma)\rightarrow \Lambda$ would be trivial.

\medskip

\section{Further comments}


It is easy to prove that the exact sequence of Theorem \ref{HR-sequence} is  contravariant  with respect to continuous $\Gamma$-maps $X\to X'$ between compact metric finite dimensional spaces. In particular, we have a functoriality morphism. 
$$
\maS_i (\Gamma) \longrightarrow \maS_i (\Gamma, X),
$$
which fits in a morphism of exact sequences. More precisely,  $C^*(\Gamma)$ is a $C^*$-subalgebra of $C^*(\maG)$ so that any Hilbert module $\maE$ over $C^*(\Gamma)$ gives rise to the Hilbert module $\maE\otimes_{C^*(\Gamma)} C^*(\maG)$ over $C^*(\maG)$. The induced $C^*$-algebra morphism $T\mapsto T\otimes_{C^*(\Gamma)} \id$ between the adjointable operators, verifies all the needed axioms to deduce  the above functoriality morphism.
This is compatible with the group morphism $K_i(C^*(\Gamma))\to K_i(C^*(\maG))$ induced  by the inclusion. This latter in turn is compatible with the Baum-Connes map and the  morphism $RK_{i, \Gamma} ({\underline{E}}\Gamma) \to RK_{i, \Gamma} ({\underline{E}}\Gamma, X)$ induced by the $\Gamma$-map $X\to \{\bullet\}$. Therefore, for any admissible completion, we end up with a morphism between the  Higson-Roe sequence for $\Gamma$ and the Higson-Roe sequence for $\maG=X\rtimes\Gamma$.

Another comment applies to the stabilizers of the $\Gamma$-action on $X$. If we fix some element $x$ of $X$ and  consider the  representation $\lambda_x$ of $C_r^*(\maG)$ in $\ell^2(\Gamma)$  defined in Section \ref{prelim} which can also be seen as a representation of any admissible completion $C^*(\maG)$, then we obtain a $C^*$-algebra morphism from  $C^*(\maG)$ to  $C^*_r(\Gamma (x))\otimes \maK (\ell^2(\Gamma/\Gamma(x)))$, where $C^*_r(\Gamma (x))$ is the reduced $C^*$-algebra of the stabilizer group $\Gamma (x)$. It is a straightforward exercise to check that this allows to deduce a morphism between the  Higson-Roe sequence for $\maG$ and the reduced Higson-Roe sequence for $\Gamma (x)$.


Next,  when $\maG$ is a torsion-free groupoid, meaning that  the stabilizers of the action of $\Gamma$ on $X$ are torsion free,  then any proper action of $\maG$ is a free  action and one can take for ${\underline{E}}\maG$ a usual classifying space $E\maG$, so that the quotient $B\maG=E\maG/\maG$ is the classifying space (unique up to homotopy) for $\maG$-principal bundles. Then, we get the following simplified version of our Theorem \ref{HR-sequence} in the torsion-free case:

\begin{corollary}\label{Torsion-free}
Assume that $\maG$ is  torsion free, then there exists for any admissible crossed product,  a  corresponding  Higson-Roe exact sequence:
$$\xymatrix{
K_{0} (B\maG) \ar@{->}[r]^{\mu_{0, \maG}\hspace{0,3cm}}\ar@{<-}[d] &
K_0(C(X)\rtimes\Gamma)\ar@{->}[r]&
\maS_1(\Gamma, X)\ar@{->}[d]\\
\maS_0(\Gamma, X)\ar@{<-}[r]&
K_1(C(X)\rtimes\Gamma)\ar@{<-}[r]^{\hspace{0,3cm}\mu_{1, \maG}}&
K_{1} (B\maG) 
}$$
\end{corollary}

We have denoted here $K_i(B\maG)$ the compactly supported $K$-homology of $B\maG$. Corollary \ref{Torsion-free} is in particular true when the group $\Gamma$  itself is torsion-free. In this case, the exact sequences corresponding to the specific space $X=\{\bullet\}$ for the reduced crossed products, were introduced and studied by  Higson-Roe  in \cite{HigsonRoeAnalysis1, HigsonRoeAnalysis2, HigsonRoeAnalysis3} while the maximal one was introduced and studied in \cite{HigsonRoe2008}. Hence, since in this cas,
we can take $B\maG$ to be $E\Gamma\times_\Gamma X$ where $E\Gamma$ is any contractible free and proper  $\Gamma$-space, there is a well defined projection map $B\maG\to B\Gamma=E\Gamma/\Gamma$.

Back to the general case, notice that when $\Gamma$ acts amenably on $X$, then  all Baum-Connes assembly maps are the same and are known to be isomorphisms by \cite{HigsonKasparov}. Therefore, in this case the universal structure groups $\maS_i(\Gamma, X)$ vanish while $\maS_i (\Gamma)$ might be non-trivial. Intereseting examples arise with residually finite exact groups which admit amenable actions on metric finite dimensional compact spaces. 
Also, for  some hyperbolic groups, the  exact sequence applies to the functorial action on their Gromov boundary $X=\partial \Gamma$ which is then metric and finite dimensional. An interesting case occurs for fundamental groups of hyperbolic manifolds, where the boundary is a sphere, the group is torsion-free  and the action is amenable, hence here the assembly map for $\maG=\partial\Gamma\rtimes\Gamma$ is known to be an isomorphism for all admissible completions. In this case, there is of course no invariant measure, and any invariants that one can construct in connection with rho invariants as pairings with the analytical structure group are automatically trivial, compare with \cite{HigsonRoe2008}. 
For residually finite non-exact groups $\Gamma$ (now we know such groups exist!), that $X\rtimes \Gamma$ satisfies  the (rectified)  Baum-Connes conjecture is still a reasonnable assumption, see \cite{HigsonLafforgueSkandalis} and \cite{BaumGuentnerWillett}.  These considerations allow in principle to deduce rigidity results for APS rho invariants as well as Cheeger-Gromov rho invariants for $\Gamma$-equivariant families of Dirac-type operators,  see \cite{HigsonRoe2008} and \cite{BenameurRoyJFA} for the case $X=\{\bullet\}$. 

Finally, let us explain the link between our constructions and  the wrong way functoriality maps  defined by Hilsum and Skandalis for foliations in \cite{HilsumSkandalis} and recently extended  in \cite{Zenobi} to exact sequences associated with adiabatic extension of pseudodifferential operators on smooth manifolds. According to \cite{HilsumSkandalis}, we need to use the maximal crossed product completions $\rtimes_{\max}\Gamma$. If we assume that $X$ is a smooth closed even dimensional oriented manifold and that $\Gamma$ acts by smooth orientation preserving diffeomorphisms, then for any smooth closed manifold $M$ with $\pi_1(M)=\Gamma$, we may consider the usual suspension of this action, say the foliated bundle $V=\tM\times_\Gamma X\to M$ where $\tM$ is the universal cover of $M$, with the  leaves being given by the projections of the submanifolds $\tM\times \{x\}$, so are transverse to the fibers. If $\Gamma_0 =\pi_1(V)$ then there exists a normal subgroup $\Gamma'$ of $\Gamma$ such that $\Gamma_0/\Gamma' \simeq \Gamma$. The monodromy groupoid of the foliation of $V$ is then Morita equivalent to our groupoid $\maG=X\rtimes \Gamma$, while the monodromy groupoid of the trivial foliation of $V$ with one leaf is Morita equivalent to $\Gamma_0$. According to \cite{HilsumSkandalis}, which can easily be extended to monodromy as opposed to  holonomy, we have a well defined Hilsum-Skandalis class  in $KK (C(X)\rtimes_{\max} \Gamma, C_{\max}^*(\Gamma_0))$ representing  Connes' transverse fundamental class in $K$-theory and precisely well defined since we are using the maximal completion $C(X)\rtimes_{\max} \Gamma$. Using the functoriality morphism $C_{\max}^*(\Gamma_0)\to C_{\max}^*(\Gamma)$ we can pushforward this class to get a class, denoted $[V/F]$ in $KK (C(X)\rtimes_{\max} \Gamma, C_{\max}^*(\Gamma))$. 
Our constructions allow, in principle,   to similarly construct   a group morphism at the level of the $K$-theories of the dual $D^*$-algebras associated with $\tM$. Hence,  one would be able to  maps the maximal Higson-Roe sequence for $X\rtimes \Gamma$ acting on $\tM\times X$ to the maximal Higson-Roe sequence for the group $\Gamma$ acting on $\tM$.

More explicit applications in relation with the rigidity problems of rho-invariants will be investigated in the forthcoming paper \cite{BenameurPreprint}.

\vspace{1\baselineskip}

\vspace{1\baselineskip}

\bibliographystyle{alpha}

\begin{thebibliography}{BHW14}


\bibitem[An02]{Anantharaman} C. Anantharaman-Delaroche {\em{Amenability and exactness for dynamical systems and their C*-algebras}} Trans. Amer. Math. Soc. 354 (2002), no. 10, 4153--4178.
 
 
\bibitem[A76]{AtiyahCovering} M. F. Atiyah 
{\em Elliptic operators, discrete groups and von Neumann algebras}. Colloque "Analyse et Topologie'' en l'Honneur de Henri Cartan (Orsay, 1974), pp. 43-72. Ast\'erisque, No. 32-33, Soc. Math. France, Paris, 1976.

\bibitem[APS75]{APS1} M. F. Atiyah, V. K. Patodi and I. M. Singer, 
 {\em Spectral asymmetry and Riemannian geometry}. I. Math. Proc. Cambridge Philos. Soc. 77 (1975), 43–69.
 

 \bibitem[APS76]{APS3} M. F. Atiyah, V. K. Patodi and I. M. Singer, {\em Spectral asymmetry and Riemannian geometry.} III. Math. Proc. Cambridge Philos. Soc. 79 (1976), no. 1, 71-99. 

\bibitem[AS63]{AtiyahSinger63}M. F. Atiyah and I. M. Singer, 
{\em The index of elliptic operators on compact manifolds,} 
 Bull. AMS {\bf 69} (1963) 422--433.
%

\bibitem[AS68a]{AtiyahSinger1} M. F. Atiyah and I. M. Singer,   
  {\em  The index of elliptic operators I},
 Ann.  of Math. {\bf 87} (1968) 484--530.

\bibitem[AS68c]{AtiyahSinger3} M. F. Atiyah and I. M. Singer,   
  {\em  The index of elliptic operators. III}. Ann. of Math. (2) 87 (1968), 546–604.

\bibitem[AS71b]{AtiyahSinger5} M. F. Atiyah and I. M. Singer,  {\em The index of elliptic operators.} V. Ann. of Math. (2) 93 (1971), 139-149.

\bibitem[AS71a]{AtiyahSinger4} M. F. Atiyah and I. M. Singer,  
 {\em  The index of elliptic operators IV},
  Ann.  of Math. {\bf 93} (1971) 119--138.

\bibitem[BC:88]{BaumConnes} P. Baum and A. Connes, {\em K-theory for discrete groups. Operator algebras and applications}, Vol. 1, 1-20, London Math. Soc. Lecture Note Ser., 135, Cambridge Univ. Press, Cambridge, 1988.

\bibitem[BCH94]{BaumConnesHigson} P. Baum, A. Connes and N. Higson
{\em Classifying space for proper actions and K-theory of group C*-algebras.(English summary)} C*-algebras: 1943–1993 (San Antonio, TX, 1993), 240--291. Contemp. Math., 167
American Mathematical Society, Providence, RI, 1994.



\bibitem[BD82]{BaumDouglas}
{\em  Index theory, bordism, and K-homology. Operator algebras and K-theory} (San Francisco, Calif., 1981), pp. 1-31, Contemp. Math., 10, Amer. Math. Soc., Providence, R.I., 1982.


\bibitem[BGW16]{BaumGuentnerWillett} P. Baum, E. Guentner and R. Willett
{\em  Expanders, exact crossed products, and the Baum-Connes conjecture}  
Ann. K-Theory 1 (2016), no. 2, 155--208.

\bibitem[B97]{BenameurPacific}
M.-T. Benameur.
\newblock Triangulations and the stability theorem for foliations.
\newblock {\em Pacific J. Math.}, 179(2):221--239, 1997.

\bibitem[B21]{BenameurCovering}
M.-T. Benameur. {\em The relative $L^2$ index theorem for Galois coverings}, arXiv:2009.10011. Ann. K-Theory 6 (2021), no. 3, 503--541.


\bibitem[B24]{BenameurPreprint}
M.-T. Benameur. {\em The universal geometric versus analytic $\ell^2$ structure group  for actions on compact spaces}, work in progress.

\bibitem[BP09]{BenameurPiazza} M-T. Benameur and P. Piazza, {\em Index, eta and rho invariants on foliated bundles}, Ast\'{e}risque 327, 2009, p.199-284.

\bibitem[BR15a]{BenameurRoyPoincare} M-T. Benameur and I.  Roy,  {\em  Leafwise homotopies and Hilbert-Poincaré complexes I. Regular HP complexes and leafwise pull-back maps}. J. Noncommut. Geom. 8 (2014), no. 3, 789-836. 

\bibitem[BR15b]{BenameurRoyJFA} M-T. Benameur and I.  Roy,  {\em The Higson-Roe exact sequence and $L^2$ eta invariants}. J. Funct. Anal. 268 (2015), no. 4, 974-1031. 

\bibitem[BR20a]{BenameurRoyI} M-T. Benameur and I.  Roy,  {\em The Higson-Roe sequence for étale groupoids. I. Dual algebras and compatibility with the BC map}. J. Noncommut. Geom. 14 (2020), no. 1, 25-71.

\bibitem[BR20b]{BenameurRoyII} M-T. Benameur and I.  Roy,  {\em The Higson–Roe sequence for étale groupoids. II. The universal sequence for equivariant families}. J. Noncommut. Geom. Electronically published on February 8, 2021. doi: 10.4171/JNCG/394 (to appear in print)

\bibitem[BR22]{BenameurRoyPPV} M-T. Benameur and I.  Roy,  {\em An equivariant PPV theorem and Paschke–Higson duality}. Ann. K-Theory 7 (2) 237 - 278, 2022. https://doi.org/10.2140/akt.2022.7.237 

\bibitem[B86]{Blackadar86}
B. Blackadar, {\em K-theory for operator algebras}, Mathematical Sciences Research Institute Publications, 5. Springer-Verlag, New York, 1986.

\bibitem[Bl96]{Blanchard96}
E. Blanchard, {\em Déformations de C*-algèbres de Hopf}, Bull. Soc. Math. France 124 (1996),
no. 1, 141-215.
 
 
 
 \bibitem[BG13]{BrownGuentner} N. Brown and E. Guentner {\em{New  C*-completions of discrete groups and related spaces}}
Bull. Lond. Math. Soc. 45 (2013), no. 6, 1181--1193.
 
 \bibitem[CW03]{ChangWeinberger} S. Chang and S. Weinberger {\em{On invariants of Hirzebruch and Cheeger-Gromov}}
Geom. Topol. 7 (2003), 311--319.
 
 
  \bibitem[CG85]{CheegerGromov} J. Cheeger and M. Gromov {\em{Bounds on the von Neumann dimension of  L2 -cohomology and the Gauss-Bonnet theorem for open manifolds}}
J. Differential Geom. 21 (1985), no. 1, 1--34.
 
 \bibitem[C79]{ConnesIntegration}
A.~Connes.
\newblock Sur la th{\'e}orie non commutative de l’int{\'e}gration.
\newblock In {\em Alg{\`e}bres d’op{\'e}rateurs}, Lecture notes in
  mathematics, pages 19--143. Springer, 1979.

 \bibitem[C82]{ConnesSurvey}
A.~Connes.
{\em A survey of foliations and operator algebras.Operator algebras and applications, Part 1} (Kingston, Ont., 1980), pp. 521--628
Proc. Sympos. Pure Math., 38
American Mathematical Society, Providence, RI, 1982


 \bibitem[CS84]{ConnesSkandalis}
A.~Connes and G.~Skandalis,
{\em The longitudinal index theorem for foliations},
Publ. Res. Inst. Math. Sci. Kyoto {\bf 20} (1984) 1139--1183.

\bibitem[D69]{Dixmier}
J. Dixmier, {\em Les C*-algèbres et leur représentations}, Editions Jacques Gabay, Paris, 1996. 403 pp. ISBN: 2-87647-013-6.

 \bibitem[BEW18]{EchterhoffWillett}
 A. Buss, S. Echterhoff and R. Willett  {\em{Exotic crossed products and the Baum-Connes conjecture}}
J. Reine Angew. Math. 740 (2018), 111--159.
 
 \bibitem[Gi94]{Gilkey} P. Gilkey {\em{Invariance theory, the heat equation and the Atiyah-Singer index theorem}},
Cambridge University Press (1994).

\bibitem[Go09]{Goehle} 
G. Goehle, {\em Groupoid crossed products}, Ph.D. thesis, Dartmouth College, 2009, arXiv:0905.4681v1.



\bibitem[HK01]{HigsonKasparov} N. Higson and G. Kasparov {\em{E -theory and  KK -theory for groups which act properly and isometrically on Hilbert space.}}
Invent. Math. 144 (2001), no. 1, 23--74.


\bibitem[HLS02]{HigsonLafforgueSkandalis} N. Higson, V. Lafforgue and G. Skandalis
{\em{Counterexamples to the Baum-Connes conjecture.}} Geom. Funct. Anal.12(2002), no.2, 330--354.

\bibitem[HR00]{HigsonRoe}
N. Higson and J. Roe {\em Analytic K-homology}. Oxford Mathematical Monographs,
Oxford Science Publications, Oxford University Press, Oxford, 2000. Zbl 0968.46058
MR 1817560

\bibitem[HR05a]{HigsonRoeAnalysis1}
N. Higson and J. Roe {\em Mapping surgery to analysis. I. Analytic signatures}. K-Theory 33 (2005), no. 4, 277-299.

\bibitem[HR05b]{HigsonRoeAnalysis2}
N. Higson and J. Roe {\em Mapping surgery to analysis. II. Geometric signatures}. K-Theory 33 (2005), no. 4, 301-324. 

\bibitem[HR05c]{HigsonRoeAnalysis3}
N. Higson and J. Roe {\em Mapping surgery to analysis. III. Exact sequences}. K-Theory 33 (2005), no. 4, 325-346. 

\bibitem[HR10]{HigsonRoe2008}
N. Higson and J. Roe {\em K-homology, assembly and rigidity theorems for relative eta invariants}. Pure Appl. Math. Q. 6 (2010), no. 2, Special Issue: In honor of Michael Atiyah and Isadore Singer, 555-601.

\bibitem[HS87]{HilsumSkandalis} 
M. Hilsum and G. Skandalis {\em{Morphismes  K -orientés d'espaces de feuilles et fonctorialité en théorie de Kasparov (d'après une conjecture d'A. Connes)
K -oriented morphisms of spaces of leaves and functoriality in Kasparov theory (after a conjecture of A. Connes)}}
Ann. Sci. Ecole Norm. Sup. (4) 20 (1987), no. 3, 325--390.


\bibitem[J02]{Julg} P. Julg {\em{La conjecture de Baum-Connes \`a coefficients pour le groupe  Sp(n,1).
The Baum-Connes conjecture with coefficients for the group  Sp(n,1)}}
C. R. Math. Acad. Sci. Paris 334 (2002), no. 7, 533--538.



\bibitem[K88]{Kasparov}
G. G. Kasparov, {\em Equivariant KK-theory and the Novikov conjecture}, Invent. Math., 91
(1988), no. 1, 147–201. Zbl 0647.46053 MR 918241


\bibitem[K80]{KasparovStinespring}
G. G. Kasparov, {\em Hilbert C*-modules: theorems of Stinespring and Voiculescu}. J. Operator Theory 4 (1980), no. 1, 133-150.


\bibitem[K00a]{Keswani1} N. Keswani {\em{Relative eta-invariants and  C*-algebra  K -theory}}
Topology 39 (2000), no. 5, 957--983.

\bibitem[K00b]{Keswani2} N. Keswani {\em{NavinVon Neumann eta-invariants and  C*algebra  K -theory}}
J. London Math. Soc. (2) 62 (2000), no. 3, 771--783.


\bibitem[KW99]{KirchbergWassermann} E. Kirchberg and S. Wassermann {\em{Exact groups and continuous bundles of  C*-algebras}},
Math. Ann. 315 (1999), no. 2, 169--203.




\bibitem[L10]{Lafforgue} V. Lafforgue {\em{Propriété (T) renforc\'ee et conjecture de Baum-Connes.
Strengthened Property (T) and Baum-Connes conjecture}}
Clay Math. Proc., 11
American Mathematical Society, Providence, RI, 2010, 323--345.

\bibitem[L95]{Lance}
E.Lance, {\em Hilbert $C^*$-modules: a toolkit for operator algebraists}, London Mathematical Society Lecture Note Series, 210. Cambridge University Press, Cambridge, 1995. x+130 pp.
ISBN: 0-521-47910-X.

\bibitem[L74]{Lawson} H.  B. Lawson, Jr.  {\em Foliations},  Bull. Amer. Math. Soc. {\bf 80} (1974) 369-418.

\bibitem[LM89]{LawsonMichelsohn} H.  B. Lawson, Jr. and  M.-L. Michelson, 
{\em Spin geometry},
 Princeton Math.  Series {\bf 38}, Princeton (1989).


 \bibitem[M92]{Mathai} V. Mathai {\em{Spectral flow, eta invariants, and von Neumann algebras}}
J. Funct. Anal. 109 (1992), no. 2, 442--456.

\bibitem[MoPhD24]{VictorThesis}
V. Moulard {\em Th\`ese de doctorat de l'universit\'e de Montpellier.} 2024.



\bibitem[MRW87]{MRW} 
P. S. Muhly, J. N. Renault and D. P.Williams, {\em Equivalence and isomorphism for groupoid C*-algebras}. J. Operator Theory 17 (1987), no. 1, 3-22.



\bibitem[N79]{Neumann} W. D. Neumann {\em{Signature related invariants of manifolds. I. Monodromy and  $\gamma$-invariants}} 
Topology 18 (1979), no. 2, 147--172.


\bibitem[0s20]{Osajda} D. Osajda, 
{\em{Small cancellation labellings of some infinite graphs and applications.}} Acta Math.225(2020), no.1, 159--191.

\bibitem[PS07]{PiazzaSchick} 
P. Piazza and T. Schick, {\em Bordism, rho-invariants and the Baum-Connes conjecture}. J. Noncommut. Geom. 1 (2007), no. 1, 27-111.

\bibitem[PS14]{PiazzaSchick2014} 
P. Piazza and T. Schick, {\em Rho-classes, index theory and Stolz' positive scalar curvature sequence}. J. Topol. 7 (2014), no. 4, 965-1004.


\bibitem[PPV79]{PPV1} 
M. Pimsner, S. Popa and D. Voiculescu
{\em Homogeneous C*-extensions of $C(X)\otimes K(H)$. I.} J. Operator Theory1(1979), no.1, 55--108.

\bibitem[PPV79]{PPV2} 
M. Pimsner, S. Popa and D. Voiculescu
{\em Homogeneous C*-extensions of $C(X)\otimes K(H)$. II.}
J. Operator Theory 4 (1980), no. 2, 211--249.




 \bibitem[Re80]{RenaultBook} J. Renault {\em A groupoid approach to C*-algebras.}
Lecture Notes in Math., 793
Springer, Berlin, 1980. ii+160 pp. ISBN:3-540-09977-8.

 \bibitem[Re87]{Renault87} J. Renault, {\em Representation des produits croises d'algebres de groupoides}, J. Operator Theory 18 (1987), no. 1, 67--97.

 \bibitem[Ro91]{RoeRelative} J. Roe, {\em A note on the relative index theorem}. Quart. J. Math. Oxford Ser. (2) 42 (1991), no. 167, 365-373.

\bibitem[Ro93]{RoeBook} J. Roe {\em Coarse cohomology and index theory on complete Riemannian manifolds}. Mem. Amer. Math. Soc. 104 (1993), no. 497, x+90 pp.

\bibitem[T99]{T99}
J-L. Tu, {\em La conjecture de Novikov pour les feuilletages hyperboliques}, K theory 16: 129-184, 1999

\bibitem[WO93]{WO}
N. E. Wegge-Olsen, {\em K-Theory and C*-Algebras (A Friendly Approach)}. Oxford
University Press, 1993.


 \bibitem[W88]{Weinberger88} S. Weinberger {\em{Homotopy invariance of rho-invariants}}, Proc. Nat. Acad. Sci. U.S.A. 85 (1988), 5362--5363. 
 
 
\bibitem[WY15]{WeinbergerYu}    
S. Weinberger and G. Yu, {\em Finite part of operator K-theory for groups finitely embeddable into Hilbert space and the degree of nonrigidity of manifolds}. Geom. Topol. 19 (2015), no. 5, 2767-2799.

\bibitem[WXY21]{WeinbergerXiYu} 
S. Weinberger, Z. Xie and G. Yu {\em{
Additivity of higher rho invariants and nonrigidity of topological manifolds.}}
Comm. Pure Appl. Math. 74 (2021), no. 1, 3--113.

\bibitem[W07]{Williams}
D. P. Williams, {\em Crossed products of $C^*$-algebras}. Mathematical Surveys and
Monographs, vol. 134, American Mathematical Society, Providence, RI, 2007.

\bibitem[W16]{Williams16}
D. P. Williams, {\em Haar systems on equivalent groupoids}, Proc. Amer. Math. Soc. Ser. B 3 (2016), 1-8.
 
 \bibitem[XY14a]{XieYu2014} 
 Z. Xie and G. Yu, {\em A relative higher index theorem, diffeomorphisms and positive scalar curvature}. Adv. Math. 250 (2014), 35-73.
 
 
  \bibitem[XY14b]{XieYu2014b} 
 Z. Xie and G. Yu, {\em Positive scalar curvature, higher rho invariants and localization algebras}. Adv. Math. 262 (2014), 823–866.
 
 \bibitem[Ze23]{Zenobi} V. F. Zenobi {\em{Interior Kasparov products for rho-classes on Riemannian foliated bundles}}.
J. Funct. Anal. 284 (2023), no. 9, Paper No. 109863, 35 pp.

\bibitem[Ze17]{ZenobiTop} V. F. Zenobi {\em{Mapping the surgery exact sequence for topological manifolds to analysis}}. J. Topol. Anal., 9(2):329--361, 2017.

 \bibitem[Zi]{Zimmer} R. J. Zimmer {\em{Ergodic theory and semi-simple groups}}. Zimmer, 
Monogr. Math., 81 Birkhäuser Verlag, Basel, 1984. x+209 pp.

\end{thebibliography}

\end{document}